\documentclass{article}
\usepackage{graphicx}
\usepackage{caption}
\usepackage{subcaption}
\RequirePackage[colorlinks,citecolor=blue,urlcolor=blue]{hyperref}

\usepackage[section]{placeins}
\usepackage[toc,page]{appendix}

\usepackage{enumitem}
\usepackage{color}
\usepackage{colordvi}
\usepackage{mathrsfs}
\usepackage{amsmath,amsfonts,amsthm, amssymb}

\usepackage{lineno}

\def\var{\mathop{\rm var}\nolimits}

\def\Cref{\ref}
\def\PY{\mathop{\rm PY}\nolimits}

\mathchardef\given="626A
\def\X{{\cal X}}
\def\A{{\cal A}}

\def\E{\EE}

\def\EE{\mathbb{E}}
\def\NN{\mathbb{N}}
\def\RR{\mathbb{R}}

\def\ra{\rightarrow}
\def\s{\sigma}
\def\a{\alpha}
\def\b{\beta}
\def\d{\delta}

\def\e{\epsilon}

\def\g{\gamma}
\def\Beta{\mathop{\rm B}\nolimits}

\let\weak\rightsquigarrow
\def\prob{\,\hbox{$\stackrel{\mbox{\scriptsize P}}{\ra}$}\, }

\newcommand{\ind}{\, \raise-2pt\hbox{$\stackrel{\mbox{\scriptsize ind}}{\sim}$}\, }
\newcommand{\iid}{\, \raise-2pt\hbox{$\stackrel{\mbox{\scriptsize iid}}{\sim}$}\,}
\def\q{\theta}
\theoremstyle{plain}
\newtheorem{theorem}{Theorem}

\newtheorem{lemma}[theorem]{Lemma}

\newtheorem{example}{Example}

\begin{document}

\title{Empirical and Full Bayes estimation of the type of a Pitman-Yor process}
\author{S.E.M.P. Franssen and A.W. van der Vaart}

\maketitle

\begin{abstract}
The Pitman-Yor process is a random discrete probability distribution of which the atoms can be used
to model the relative abundance of species. The process is indexed by a type parameter $\s$, which controls the
number of different species in a finite sample from a realization of the distribution. 
A random sample of size $n$ from the Pitman-Yor process of
type $\s>0$ will contain of the order $n^\sigma$ distinct values (``species'').
In this paper we consider the estimation of the type parameter by both empirical Bayes and full Bayes methods.
We derive the asymptotic normality of the empirical Bayes estimator and 
a Bernstein-von Mises theorem for the full Bayes posterior, in the frequentist setup that the observations
are a random sample from a given true distribution. We also consider the estimation of the
second parameter of the Pitman-Yor process, the prior precision.
We apply our results to derive the limit behaviour of the likelihood ratio in a setting of
forensic statistics. 
\end{abstract}

\section{Introduction}
The Pitman-Yor process \cite{PitmanandYor(1997),PermanPitmanandYor(1992)}
is a random discrete probability distribution, which can be used as a model for the relative abundance of species.
It is characterised by a \emph{type} parameter $\sigma$. Our main aim
is statistical inference on this type parameter.

The Pitman-Yor process of type $\sigma=0$
is the Dirichlet process \cite{Ferguson(1974)}, which is well understood, while negative types correspond to finitely discrete distributions
and were considered in \cite{DeBlasiLijoiandPrunster(2013)}. In this paper we concentrate on Pitman-Yor processes of positive type.
The Pitman-Yor process is also known as the two-parameter Poisson-Dirichlet Process,
 is an example of a Poisson-Kingman process \cite{Pitman(2003)}, and a species sampling process
of Gibbs type~\cite{deBlasi2015}.

The easiest definition is through \emph{stick-breaking} (\cite{PermanPitmanandYor(1992),IshwaranJames}), as follows.
The family of nonnegative Pitman-Yor processes is given by three parameters: a number $\sigma \in [0,1)$, a number $M > -\sigma$
and an atomless probability distribution $G$ on some measurable space $(\X,\A)$.
We say that a random probability measure $P$ on $(\X,\A)$ is a Pitman-Yor process (of nonnegative type), 
denoted $P \sim \PY\left(\sigma, M, G \right)$, if $P$ can be represented as
\[
P = \sum_{i = 1}^\infty W_i \delta_{\theta_i},
\]
where $W_i = V_i \prod_{j = 1}^{i - 1} (1 - V_j)$ for $V_i $ independent variables with $V_i\sim\Beta( 1 - \sigma, M + i \sigma)$, independent of 
$\theta_i \iid G$.

It is clear from this definition that the realisations of $P$ are discrete probability measures, with countably many
atoms at random locations, with random weights. If one first draws $P\sim \PY\left(\sigma, M, G \right)$, and next given $P$ a random sample
$X_1,\ldots, X_n$ from $P$, then ties among the latter observations are likely. Each tie represents an observed species.
Each species is identified by a label $\q_i$, but this does not play an important role in the present paper, 
except that labels are unique by the assumption
that $G$ is atomless. If the weights $(W_i)$ correctly model the relative abundance of distinct species, 
then the (conditional) probability that a next observation $X_{n+1}$ will be distinct
from $X_1,\ldots, X_n$ is indicative of the likelihood of finding a new species given the past observations. 

In a forensic setup one may consider two instances not present in a ``database'' $X_1,\ldots, X_n$
and compare the hypotheses that these two values are two independent draws $X_{n+1},X_{n+2}$ from the population
or copies of a single draw $X_{n+1}$. If one of the new instances is a characteristic (e.g.\ DNA profile) found at a crime scene
and the other a characteristic of a suspect, one can so derive the ratio of the probabilities that
the characteristic at the crime scene originates from the perpetrator (two copies of a single draw $X_{n+1}$)
or that the match is by chance (two draws $X_{n+1}, X_{n+2}$ that happen to be the same). See \cite{Cereda2017,Cereda2019}
(and Section~\ref{SectionForensic}) for further discussion.

As shown in \cite{Cereda2019} such probabilities are readily calculable 
under the Pitman-Yor model, but depend strongly on the type parameter $\s$.
Thus it is desirable to estimate this parameter from the observed data. In this paper we consider both the empirical Bayes 
estimator (which is the maximum likelihood estimator in the Bayesian setup with the Pitman-Yor process viewed as a prior) 
and the full Bayes posterior of $\s$. We prove that the empirical Bayes estimator is asymptotically normal, and show that the
posterior distribution satisfies a corresponding Bernstein-von Mises theorem. We apply these results to the forensic
problem.

The Pitman-Yor process is characterised by a second parameter $M$,  called the prior precision. As its name suggests,
this is a prior modelling parameter, and perhaps is better not estimated. We show that the asymptotics of the type
parameter are almost independent of the prior precision, and that the problem of estimating the prior precision degenerates
as $n\ra\infty$.

Other applications of the Pitman-Yor process, in genetics or topic modelling,
can be found in \cite{Wood2009, Teh2006, Goldwater2005, Arbel2018,favaro2021nearoptimal}. 
The limiting case of the Pitman-Yor process with $M\downarrow0$ is a (particular) Poisson-Kingman process
(see \cite{PitmanandYor(1997)}, or Example~14.47 in \cite{FNBI}).
Estimation of the type parameter for this process was considered in \cite{favaro2021nearoptimal}, who
give conditions so that the maximum likelihood estimator converges to a limit at a certain rate.
The present paper generalizes and refines these results in the case of a general $M$ 
(by providing a sharp rate and a limit distribution), and adds a full
Bayes analysis.

The posterior distribution when the Pitman-Yor process is used as a nonparametric prior on the distribution
of the observations was studied in \cite{James2008} for known type parameter and continuous
observations, and for unknown type and general observations in \cite{Franssen2020}. In the latter
paper it is shown that the posterior distribution is not very sensitive to the type parameter,
and its asymptotics could be established knowing just consistency for this parameter,
and hence without interception of the precise results of the present paper.

\section{Main results}
We consider statistical procedures derived from the Bayesian model in which a probability distribution $P$ is 
drawn from the Pitman-Yor process, viewed as a prior over the set of probability measures, and next
given $P$ the observations $X_1,\ldots,X_n$ are an i.i.d.\ sample from $P$. To estimate the parameters
$\s$ (or $(\s,M)$) of the Pitman-Yor process, we consider the maximum likelihood estimator, based on
the marginal distribution of $(X_1,\ldots,X_n)$ in this setup, and a full Bayes approach. Because the maximum likelihood
estimator estimates a parameter of the prior, in this setup this estimator is also called an empirical Bayes 
estimator. The full Bayes approach adds a prior distribution over $\s$ (or $(\s, M)$) to the Bayesian hierarchy 
and then uses the conditional  distribution of $\s$ given $X_1,\ldots, X_n$, the posterior distribution,
for further inference. 

While these two procedures are defined by the Bayesian setup, our theoretical results obtained below are frequentist Bayes:
we study these two procedures, both functions of the observations $X_1,\ldots,X_n$, under the
assumption that these observations are a random sample from a given distribution $P_0$.

Both statistical procedures are based on the Bayesian likelihood for $X_1,\ldots,X_n$. This can be conveniently
obtained by considering the exchangeable random partition of the set $\{1, 2,\ldots, n\}$ generated by the sample
through the equivalence relation
$$i\equiv j\qquad\qquad \text{ if and only if }\qquad\qquad X_i=X_j.$$
An alternative way to generate $X_1,\ldots,X_n$ is to generate first the partition and next attach
to each set in the partition a value generated independently from the center measure $G$
(see e.g.\ \cite{FNBI}, Lemma~14.11 for a precise statement), duplicating this as many times as there are indices in the set, in order
to form the observations $X_1,\ldots, X_n$.
Because the parameter $\s$ enters only in creating the
partition, the partition is a sufficient statistic for $\s$. Because of exchangeability, the vector
$(K_n,N_{n,1},\ldots, N_{n,K_n})$ of the number $K_n$ of sets in the partition and
the cardinalities $N_{n,i}$ of the partitioning sets (i.e.\ the multiplicity of $X_i$ in $X_1,\ldots, X_n$)
 is already sufficient for $(M,\s)$ and hence
the empirical Bayes estimator and posterior distribution of $(M,\s)$ are the same, whether based on observations $(X_1,\ldots,X_n)$
or on  observations $(K_n,N_{n,1},\ldots, N_{n,K_n})$.

The likelihood function for $(M,\s)$ is therefore equal to the probability of a particular partition,
called the \emph{exchangeable partition probability functio}n (EPPF). For the Pitman-Yor process this is
given by (see \cite{Pitman1996b}, or \cite[page 465]{FNBI})
\begin{equation}
\label{EqLikelihoodSigma}
p_\s(N_{n,1},\ldots, N_{n,K_n})=\frac{\prod_{i=1}^{K_n-1}(M+i\s)}{(M+1)^{[n-1]}}\prod_{j=1}^{K_n}(1-\s)^{[N_{n,j}-1]}.
\end{equation}
Here $a^{[n]}=a(a+1)\cdots(a+n-1)$ is the ascending factorial, with $a^{[0]}=1$ by convention. 

Although we adopt the Pitman-Yor process as a prior for the distribution $P$ of the observations, and then arrive
at the likelihood \eqref{EqLikelihoodSigma}, we shall investigate the maximum likelihood estimator and
posterior distribution under the frequentist-Bayes assumption that in reality the observations $X_1,\ldots, X_n$ are an i.i.d.\ sample 
from a given distribution $P_0$. It turns out that the asymptotic properties of the maximizer of \eqref{EqLikelihoodSigma} are then determined by
the function $\a_0: (1,\infty)\to \NN$ given by 
\begin{equation}
\label{EqFirstOrderRegular}
\a_0(u):=\#\Bigl\{ x: \frac1{P_0\{x\}}\le u\Bigr\}.
\end{equation}
The function $\a_0$ is nondecreasing and increases by jumps of size 1.
Following \cite{Karlin1967}, we shall assume that this function is regularly varying. The paper
\cite{Karlin1967} derived (distributional) limits of characteristics such as $K_n$. In the present paper
we use similar methods to analyse the likelihood function \eqref{EqLikelihoodSigma}.

Recall that a measurable function $\a: (0,\infty)\to \RR_+$ is said to be \emph{regularly varying} (at $\infty$) 
of order $\g$ if, for all $u>0$, as $n\ra\infty$,
$$\frac{\a(nu)}{\a(n)}\ra u^\g.$$
It is known (see e.g.\ \cite{Binghametal} or the appendix to \cite{deHaan})
that if the limit of the sequence of quotients on the left exists for every $u$, then it necessarily has the form $u^\g$, for some $\g$.
If we write $\a(u)=u^\g L(u)$, then $L$ will be \emph{slowly varying}: a function that is regularly varying of order 0.
Then $\a(n)=n^\g L(n)$, and it can be shown that $n^{\g-\d}\ll \a(n)\ll n^{\g+\d}$, for every $\d>0$, so that the
rate of growth of $\a$ is $n^\g$ to ``first order''. 

Since the function $\a_0$ in \eqref{EqFirstOrderRegular} increases 
by steps, it is discontinuous and so is the associated slowly varying function.
We assume that there exists $\s_0\in (0,1)$ and a continuously differentiable slowly varying function $L_0$ such that,
for every $u>1$,
\begin{alignat}{4}
\bigl|\a_0(u)-u^{\s_0} L_0(u)\bigr|&\le C u^{\b_0},&&\qquad  &&\text{for some } \b_0<\s_0 \text{ and }  C>0,\label{EqApproxAlpha}\\
\bigl|L_0'(u)\bigr|&\le C_\d \,u^{-1+\d},&&\qquad &&\text{for any }\d>0, \text{ for some }C_\d.\label{EqBoundL}
\end{alignat}
If \eqref{EqApproxAlpha} holds for a slowly varying function $L_0$, then $\a_0$ is regularly varying of
order $\s_0$ (see Lemma~\ref{LemmaRV}). The rationale for \eqref{EqBoundL} is that a slowly varying function
$u\mapsto L(u)$ will always grow slower than any power $u^\d$. It may not be differentiable, but if it is,
then it is reasonable that the derivative grows slower than any power of $u^{-1+\d}$. Condition \eqref{EqBoundL} 
is satisfied, for instance, by powers  $L_0(u)=(\log u)^r$ of the logarithmic function, for $r\in\RR$, and the functions
$L_0(u)=e^{(\log u)^r}$, for $r<1$. Although we assume \eqref{EqApproxAlpha}-\eqref{EqBoundL} for $u>1$, the bounds
are asymptotic in the sense that if they hold for $u>U$ and some $U$, then they are valid for $u>1$ and possibly
larger constants $C$ and $C_\d$ (and some extension of $L_0$), since the right sides are bounded away from zero on intervals $(1,U]$.

\begin{example}
\label{ExamplePolynomialpj}
For the probability distribution $(p_j)_{j\in\NN}$ with $p_j=c/j^\a$, for some $\a>1$, we have
$\a_0(u):=\# (j: 1/p_j\le u)= \lfloor (cu)^{1/\a}\rfloor$. Then $|\a_0(u)-(c u^{1/\a})|\le 1$ and hence
 \eqref{EqApproxAlpha}-\eqref{EqBoundL} are satisfied with $\s_0=1/\a$,  $L_0(u)= c^{1/\a}$, $C=1$, $\b_0=0$, and $C_\d=0$.
\end{example}

\begin{example}
\label{ExampleFavaro}
In  \cite{favaro2021nearoptimal} it is assumed that 
$|\alpha_0(u)-Lu^{\s_0}|\le C u^{\s_0/2}\sqrt{\log(eu)}$, for every $u>1$ and some constant $C$.
This  implies \eqref{EqApproxAlpha}-\eqref{EqBoundL} with $L_0$ constant and any  $\b_0$ slightly bigger
than $\s_0/2$ (and hence easily satisfying the restriction in \eqref{EqApproxAlpha}). 
\end{example}

\subsection{Estimating the type parameter}
The following theorem shows that the empirical likelihood estimator  is asymptotically normal,
after scaling by the rate $\sqrt{\a_0(n)}$ and centering at the zero $\s_{0,n}$ of the function
\begin{equation}
E_{0,n}(\s)= \int_0^n\a_0\Bigl(\frac ns\Bigr)e^{-s}\Bigl(\frac 1\s-\sum_{m=1}^\infty \frac{s^m}{m!(m-\s)}\Bigr)\,ds.
\label{EqDefEn}
\end{equation}
In Lemma~\ref{LemmaZero} these zeros are shown to converge to the exponent $\s_0$ of regular variation of $\a_0$,
at a rate depending on the function $\a_0$.

For the moment we keep $M$ fixed and let $\hat\s_n$ be the maximizer of the likelihood \eqref{EqLikelihoodSigma} with respect to 
$\s$, for given $M$. The following theorem shows that the asymptotic behaviour of $\hat\s_n$ is the same for every $M$.

\begin{theorem}
\label{ThmSigma}
Assume that $P_0$ is discrete with atoms such that $\a_0(u):=\#\{ x: 1/P_0\{x\}\le u\}$
satisfies \eqref{EqApproxAlpha}-\eqref{EqBoundL} with exponent $\s_0\in(0,1)$. Then  the empirical Bayes estimator
$\hat\s_n$, the point of maximum of \eqref{EqLikelihoodSigma}, satisfies 
$\sqrt{\a_0(n)}(\hat\s_n-\s_{0,n})\weak N(0,\tau_1^2/\tau_2^4)$, 
where $\s_{0,n}$ are the roots of the functions $E_{0,n}$ in \eqref{EqDefEn} 
and $\tau_1$ is given in \eqref{EqDeftau1Echt} and $\tau_2^2=-E_0'(\s_0)$, for
$E_0$ given in \eqref{EqDefE}.
\end{theorem}

The proof of the theorem is deferred to Section~\ref{SectionProofThmSigma}.

The theorem centers the estimators at the zeros $\s_{0,n}$ of the functions $E_{0,n}$ in \eqref{EqDefEn}. It is shown
in Lemma~\ref{LemmaZero} that these zeros tend to the coefficient of regular variation $\s_0$ of $\a_0$, and 
hence $\hat\s_n\ra \s_0$, in probability.
However, the rate of this convergence may be too slow to replace $\s_{0,n}$ by $\s_0$ in the centering of 
$\hat\s_n$. For the case that the function $L_0$ in \eqref{EqApproxAlpha} can be taken constant, 
Lemma~\ref{LemmaZero} shows that $\s_{0,n}-\s_0=O(n^{-(\s_0-\b)})$, for $\b_0$ as in \eqref{EqApproxAlpha},
and hence if $\b_0<\s_0/2$, then $\sqrt {\a_0(n)}(\s_{0,n}-\s_0)\ra 0$,
and hence also $\sqrt{\a_0(n)}(\hat\s_n-\s_0)\weak N(0,\tau_1^2/\tau_2^4)$. 
If $\a_0$ contains nontrivial slowly varying terms, then the rate of convergence
$\s_{0,n}\ra\s_0$ will typically be much slower than $\a_0(n)^{-1/2}$ and the latter result will fail.
(Lemma~\ref{LemmaZero} gives the rate $L_0'(n)n/L_0(n)$, and its proof shows that this is
sharp, for instance: $1/\log n$ if $L_0(s)=\log s$.) 

In general, we could say that
the estimators $\hat\s_n$ roughly, but possibly not quite, estimate the degree of regular variation of $\a_0$.
If one believes in the likelihood, then this is an indication that the type parameter has a more subtle
interpretation than the degree of regular variation, rendering it extra worth while to use principled
methods for its estimation. (If interest were in the coefficient of regular variation $\s_0$, 
then direct approaches may be preferable.)
 
In the special case considered in Example~\ref{ExampleFavaro}, the rate of the estimator
is $\sqrt{\a_n(n)}=n^{\s_0/2}$. This is a faster rate than obtained (in the case that $M=0$) 
in  \cite{favaro2021nearoptimal}, who showed 
that $\hat\s_n=\s_0+O_P(n^{-{\s_0}/2}\sqrt{\log n})$ under the condition in Example~\ref{ExampleFavaro}. 
(Our improved rate centers at $\s_{0,n}$; for centering at $\s_0$, the condition 
of Example~\ref{ExampleFavaro} must be slightly strengthened
to have an exponent strictly smaller than $\s_0/2$ rather than $\s_0/2$.)

The asymptotic variance of the sequence $\sqrt{\a_0(n)}(\hat\s_n-\s_{0,n})$ 
is a one-dimen\-sional form of the sandwich formula, which is clear if it is
written as $\tau_2^{-2}\tau_1^2\tau_2^{-2}$. It appears that in general $\tau_1\not =\tau_2$, which is
explainable by the fact that the likelihood used to define $\hat\s_n$ is the Bayesian marginal likelihood,
which is misspecified relative to the frequentist distribution of $X_1,\ldots, X_n$: the likelihood behaves
like a general contrast function rather than a likelihood.

Next consider the posterior distribution of $\s$ given $X_1,\ldots, X_n$ in the model 
$\s\sim \Pi_\s$, $P\given \s\sim\PY \left( \sigma, M, G \right)$ and $X_1, \ldots, X_n \given P,\s \sim P$, for
a given prior distribution $\Pi_\s$ on $(0,1)$. Since the likelihood for observing $X_1,\ldots, X_n$
is proportional  to \eqref{EqLikelihoodSigma},  by Bayes rule the posterior distribution has density relative to $\Pi_\s$ proportional to 
\eqref{EqLikelihoodSigma}. We study the posterior distribution under the assumption that  $X_1,\ldots,X_n$ are an i.i.d.\ sample
from a distribution $P_0$. 

\begin{theorem}
\label{ThmSigmaPost}
Assume that $P_0$ is discrete with atoms such that $\a_0(u):=\#\{ x: 1/P_0\{x\}\le u\}$
satisfies \eqref{EqApproxAlpha}-\eqref{EqBoundL} with exponent $\s_0\in(0,1)$.
For a prior distribution $\Pi_\s$  on $\s\in (0,1)$ with a bounded density that is positive and continuous at $\s_0$, 
the posterior distribution of $\s$ 
satisfies
$$\sup_B\Bigl|\Pi_n(\s\in B\given X_1,\ldots, X_n)- N\Bigl(\hat\s_n,\frac1{\a_0(n)\tau_2^2}\Bigr)(B)\Bigr|\ra 0,$$
where the supremum is taken over all Borel sets $B$ in $(0,1)$, and 
$\tau_2^2=-E_0'(\s_0)$, for $E_0$ given in \eqref{EqDefE}. In particular, 
the posterior distribution for $\s$ contracts to $\s_0$. 
Furthermore, the posterior mean $\tilde\s_n=\E(\s\given X_1,\ldots, X_n)$
satisfies $\sqrt{\a_0(n)}(\tilde\s_n-\hat\s_n)\ra 0$, in probability.
\end{theorem}

The proof Theorem~\ref{ThmSigmaPost} is deferred to Section~\ref{SectionProofThmSigmaPost}. 

Apart from the unusual scaling rate $\a_0(n)$, Theorem~\ref{ThmSigmaPost} is of the Bernstein-von Mises type,
for a misspecified model (see \cite{KleijnvanderVaart}). Misspecification arises, because the likelihood corresponds
to the Bayesian model, but the observations are sampled from $P_0$.

\subsection{Estimating the precision}
The parameter $M$ is commonly referred to as the \emph{prior precision}. This name suggests 
that this is truly a prior modelling parameter and estimating it from the data may be illogical. 
Theorem~\ref{ThmSigma} shows that the maximum likelihood estimator  $\hat\s_{n,M}$ of $\s$, for a given $M$,
satisfies  that the sequence $\sqrt{\a_0(n)}(\hat\s_{n,M}-\s_{n,0})$ tends to a centered normal distribution, for every $M$,
where the limit is independent of $M$.
Inspection of the proof shows that  $\hat\s_{n,M_n}$, for a sequence $M_n$, has
the same behaviour, as long as $M_n\ll \sqrt{\a_0(n)}/\log n$. Furthermore, if $M$ is equipped with a prior
over a compact (or slowly increasing) interval, then the posterior distribution of $\s$ still satisfies the
assertion of Theorem~\ref{ThmSigmaPost}, where the limit does not involve $M$. Thus for a very wide range
of prior precisions, the estimators for the type parameter are asymptotically equivalent. This may
again suggest that the parameter $M$ plays a different role than the type parameter. 

Nevertheless, we might use the likelihood function
to obtain also a maximum likelihood or Bayes estimator for $M$. 
In the following theorem we consider the maximum likelihood estimator, where the parameter $M$ 
is restricted to a compact set $[0, \bar M]$. (The proof extends to $\bar M=\bar M_n$ that increase not too fast to infinity.)

The  limiting value $M_0$ of the maximum likelihood estimator
$\hat M_n$ depends on the fine details of the regular variation
of the function $\a_0$ in \eqref{EqApproxAlpha}-\eqref{EqBoundL}, through the  limit 
(assumed to exist)
$$K_0:=\lim_{n\ra\infty} \Bigl[\frac{c_0L_0'(n)n\log n}{L_0(n)}+\log L_0(n)\Bigr],$$
where $c_0=\Gamma(1-\s_0)(1+\s_0)/(\s_0\tau_2^2)$.
Define $M_0=0$ or $M_0=\bar M$ if the limit is $K_0=\infty$ or $K_0=-\infty$, respectively,
and otherwise, set it equal to the maximizer in $[0,\bar M]$ of the function
$$M\mapsto \frac M{\s_0}\bigl(K_0+\log \Gamma(1-\s_0)\bigr)+\log \Gamma(1+M)-\log \Gamma\Bigl(1+\frac{M}{\s_0}\Bigr).$$

\begin{theorem}
\label{ThmM}
Assume that $P_0$ is discrete with atoms such that $\a_0(u):=\#\{ x: 1/P_0\{x\}\le u\}$
satisfies \eqref{EqApproxAlpha}-\eqref{EqBoundL} with exponent $\s_0\in(0,1)$ and a function
$L_0$ such that $u\mapsto L_0'(u)u$ is slowly varying.
Then the joint maximum likelihood estimator $(\hat M_n,\hat\s_n)$ satisfies $\hat M_n\ra M_0$ in probability
and $\hat\s_n$ has identical behaviour as in Theorem~\ref{ThmSigma}.
\end{theorem}

The proof of the theorem is deferred to Section~\ref{SectionProofThmM}.

\begin{example}
If $L_0$ is constant, then $K_0=\log L_0$ is finite, and it can be arbitrary large or small, depending
on the nonasymptotic properties of the sequence $P_0\{x_j\}$. 
For instance, the choices $p_j=c/j^\a$ for every $j>J$, and $p_j=\eta$, for $j\le J$,
are possible for every constants $c>0$, $\a>1$, $J\in \NN$, and $\eta\in (0,1)$ such that $J\eta+\sum_{j>J}c/j^\a=1$. Then
$\a_0(u)=J+\lfloor (cu)^{1/\a}-J\rfloor=\lfloor (cu)^{1/\a}\rfloor$, for $u>1/\eta$, 
and hence $L_0=c^{1/\a}$, as \eqref{EqApproxAlpha} is determined by $u\ra\infty$.
Depending on the constants $c$ and $\a$, the constant $M_0$ can be any value in $[0,\bar M]$.
\end{example}

\begin{example}
If $L_0(u)=\log u$, then $K_0=\infty$, and hence $M_0=\bar M$.
If $L_0(u)=1/\log u$, then $K_0=-\infty$, and hence $M_0=0$.
\end{example}

\subsection{Forensic application}
\label{SectionForensic}
Consider again the Bayesian model in which $X_1,\ldots, X_n$ are drawn independently from a distribution
$P$ generated from the Pitman-Yor process. A next observation $X_{n+1}$ will either be equal to 
one of the current observations $X_1,\ldots, X_n$ or constitute a new type. If $\tilde X_1,\ldots, \tilde X_{K_N}$ are
the distinct values in $X_1,\ldots, X_n$ and $N_{n,1},\ldots, N_{n,K_n}$ are the multiplicities of these values
in the sample, then it is known that (see  \cite{Pitman(1995),Pitman(2003)}, or \cite{FNBI}, page 465)
\begin{equation}
\label{EqPPF}
\Pr(X_{n+1}=\tilde X_i\given X_1,\ldots, X_n,\s, M)= \frac{N_{n,i}-\s}{M+n},\qquad i=1,\ldots, K_n.
\end{equation}
The remaining mass  $(M+K_n\s)/(M+n)$ is the probability of obtaining a new species, distinct from
$\tilde X_1,\ldots, \tilde X_{K_N}$.
The probability distribution defined by these numbers is known as the \emph{prediction probability function}.

In the forensic setup discussed in \cite{Cereda2017,Cereda2019}, the sample $X_1,\ldots, X_n$ represents
a database of characteristics of individuals (say DNA profiles), and given is a new profile, 
not present in the database, that has been found both at the crime scene and on a defendant who has been charged
with the crime. The prosecution argues that the defendant is the perpetrator who left her profile on the crime scene 
and hence only a single new observation $X_{n+1}$ is involved. The defence argues that the perpetrator
and the defendant are two different individuals, who happen to have the same profile, and hence two independent 
observations $X_{n+1},X_{n+2}$ are involved, which were observed to take the same value. 
The two hypotheses can be made precise in a Bayesian hierarchy
describing a generative model for the database $X_1,\ldots, X_n$, the profile $X_{n+1}$ found at the crime scene
and the profile $Y$ found on the suspect. According to the prosecution the generative model is:
\begin{itemize}
\item[(i)] $(\s, M)\sim \Pi_{\s,M}$.
\item[(ii)] $P\given (\s, M)\sim \text{Pitman-Yor }(\s,M)$.
\item[(iii)] $X_1,\ldots,X_{n+1}\given P,\s, M\iid P$.
\item[(p)] $Y\given X_1,\ldots,X_{n+1}, P,\s, M\sim \d_{X_{n+1}}$.
\end{itemize}
The fourth step (p) expresses that the profile $Y$ found on the defendant is identical to the profile found on the crime scene,
because it results from the same individual $X_{n+1}$ chosen from the population.
The defence agrees with steps (i)-(iii) of the hierarchy, but replaces (p) by:
\begin{itemize}
\item[(d)] $Y\given X_1,\ldots,X_{n+1}, P,\s, M\sim P$.
\end{itemize}
This expresses that the defendant's profile is just another draw $X_{n+2}$ from the population. 
In the observed data the value of this draw happens to be the same as the profile $X_{n+1}$ at the crime scene.

To decide on the case we might evaluate the ratio
of the likelihoods of the full evidence $X_1,\ldots,X_{n+1},Y$ under the two hypotheses. Since the (marginal) likelihood of 
$X_1,\ldots,X_{n+1}$ is determined by (i)--(iii), it is the same under both hypotheses. Hence  the relative likelihood
is the ratio of the conditional likelihoods of $Y$ given $X_1,\ldots, X_{n+1}$. For the prosecution $Y$ 
depends deterministically on $X_1,\ldots, X_{n+1}$ (it must be equal to $X_{n+1}$) and hence the conditional probability of the observed value is
equal to 1. For the defence the conditional likelihood of $Y$ is the probability that an $(n+2)$th observation $X_{n+2}$
happens to be of the same species as $X_{n+1}$. Since $X_{n+1}$ has multiplicity $N_{n+1,K_{n+1}}=1$  among $X_1,\ldots, X_{n+1}$
(by our assumption on the observed data), given $(\s,M)$ this probability is $(1-\s)/(M+n+1)$, as determined by the
prediction probability function \eqref{EqPPF}. The unconditional probability is the integral of this 
 relative to the posterior distribution of $(\s,M)$,
i.e.\ $\EE\bigl((1-\s)/(M+n+1)\given X_1,\ldots, X_{n+1}\bigr)$. The likelihood ratio of prosecution versus defence is
therefore
$$1\ \Big/\  \ \EE\Bigl(\frac{1-\s}{n+1+M}\given X_1,\ldots, X_{n+1}\Bigr).$$

\begin{theorem}
\label{ThmForensic}
Under the assumptions of Theorem~\ref{ThmSigmaPost}, the posterior distribution of $\phi=(1-\s)/(M+n+1)$ satisfies
$$\sup_B\Bigl|\Pi_n(\phi\in B\given X_1,\ldots, X_n)- N\Bigl(\frac{1-\hat\s_n}{n+M+1},\frac{1}{(n+M+1)^2\a_0(n)\tau_2^2}\Bigr)(B)\Bigr|\ra 0.$$
Moreover, the posterior mean $\tilde\phi_n=\EE\bigl(\phi\given X_1,\ldots, X_{n+1}\bigr)$ satisfies
$$\sqrt{\a_n(n)}\Bigl(\frac1{n\tilde \phi_n}-\frac1{1-\s_{n,0}}\Bigr)\weak N\Bigl(0,\frac{\tau_1^2}{(1-\s_0)^4\tau_2^4}\Bigr).$$
These assertions remain true if $M$ is equipped with a prior supported on a compact interval in $[0,\infty)$.
\end{theorem}

\begin{proof}
The first assertion is immediate from Theorem~\ref{ThmSigmaPost} and the definition of $\phi$.
For the second assertion we note that 
$$\frac1{n\tilde \phi_n}-\frac1{1-\s_{0,n}}=
=\frac{\tilde\s_n-\s_{0,n}+(1-\tilde\s_n)(M+1)/(M+n+1)}{(1-\tilde\s_n)/(M+n+1)}.$$
By Theorem~\ref{ThmSigmaPost} the sequence 
$\sqrt{\a_0(n)}(\tilde\s_n-\s_{n,0})$ is asymptotically equivalent to the sequence
$\sqrt{\a_0(n)}(\hat\s_n-\s_{0,n})$, which tends to the $N(0,\tau_1^2/\tau_2^4)$-distribution,
by Theorem~\ref{ThmSigma}. The second assertion follows by Slutzky's lemma.
\end{proof}

\section{Proofs}
\label{SectionProofs}
The logarithm of the likelihood \eqref{EqLikelihoodSigma} can be written
\begin{align}
\nonumber
\Lambda_n(\s,M)&=\sum_{l=1}^{K_n-1}\log (M+l\s)+\!\!\sum_{j=1: N_{n,j}\ge 2}^{K_n}\!\!\sum_{l=1}^{N_{n,j}-1}\log (l-\s)-\sum_{i=1}^{n-1} \log (M+i)\\
&=\sum_{l=1}^{K_n-1}\log (M+l\s)+\sum_{l=1}^{n-1}\log (l-\s)Z_{n,l+1}-\sum_{i=1}^{n-1} \log (M+i),
\label{EqLogLikelihood}
\end{align}
where $Z_{n,l}=\#(1\le j\le K_n: N_{n,j}\ge l)$ is the number of distinct values of multiplicity at least $l$ in
the sample $X_1,\ldots, X_n$. (In the case that all observations are distinct and hence $N_{n,j}=1$ for every $j$,
the second term of the likelihood is equal to 0.)

For the proofs of Theorems~\ref{ThmSigma} and~\ref{ThmSigmaPost}, we fix the argument $M$ and drop it from
the notation. The concavity of the logarithm shows that the log likelihood is a strictly concave function of $\s$.
For $\s\downarrow0$, it tends to a finite value if $M>0$ and to $-\infty$ if $M=0$, while for $\s\uparrow1$ it tends to $-\infty$ if the term
with  $l=1$ is present in the second sum, i.e.\ if there is at least one tied observation. This happens
with probability tending to 1 as $n\ra\infty$. The derivative of the log likelihood with respect to $\s$ is equal to 
\begin{equation}
\label{EqLambdaprime}
\Lambda_n'(\s)=\sum_{l=1}^{K_n-1}\frac{l}{M+l\s}-\sum_{l=1}^{n-1}\frac{Z_{n,l+1}}{l-\s}.
\end{equation}
The left limit at $\s=0$ is $\Lambda_n'(0)=\frac12 K_n(K_n-1)/M-\sum_{l=1}^{n-1}l^{-1}Z_{n,l+1}$.
Since $Z_{n,l}\le Z_{n,1}=K_n$, a crude bound on the sum is $K_n\log n$, which shows that the derivative
at $\s=0$ tends to infinity if $K_n\gg \log n$. In that case the unique maximum of the log likelihood 
in $[0,1]$ is taken in the interior of the interval,  and hence $\hat\s_n$ satisfies $\Lambda_n'(\hat\s_n)=0$.

Set $\a_n=\a_0(n)$. Under the condition that $\a_0$ is regularly varying of exponent $\s_0\in(0,1)$, the sequence 
$\a_n$ is of the order $n^{\s_0}$ up to slowly varying terms.
By Theorems~9 and~1` of \cite{Karlin1967},
the sequence $K_n/\a_n$ tends almost surely to $\Gamma(1-\s_0)$ and hence in particular $K_n\gg \log n$,
and the conclusion of the preceding paragraph that $\Lambda_n'(\hat\s_n)=0$ pertains.

By Lemma~\ref{LemmaZero},  the functions $E_{0,n}/\a_n$, for 
 $E_{0,n}$ defined in \eqref{EqDefEn}, converge to the function $E_0$ defined by
\begin{equation}
\label{EqDefE}
E_0(\s)=\frac{\Gamma(1-\s_0)}\s-\sum_{m=1}^\infty \frac{\Gamma(m+1-\s_0)}{m!(m-\s)},
\end{equation}
and this function vanishes at $\s=\s_0$. By monotonicity of these functions, the
zeros $\s_{0,n}$ of the functions $E_{0,n}$ tend to the zero $\s_0\in(0,1)$ of the limit function.

\subsection{Proof of Theorem~\ref{ThmSigma}}
\label{SectionProofThmSigma}
The monotonicity of $\Lambda_n'$, the definition of $\hat\s_n$ and the fact that $-E_{0,n}(\sigma_{0,n}) = 0$ give that 
\begin{align*}
&\Pr\bigl(\sqrt {\a_n}(\hat\s_n-\s_{0,n})\le x\bigr)
=\Pr\Bigl(\Lambda_n'\Bigl(\s_{0,n}+\frac{x}{\sqrt{\a_n}}\Bigr)\le 0\Bigr)\\
&\qquad =\Pr\biggl(\frac{1}{\sqrt{\a_n}}\Bigl[\Lambda_n'\Bigl(\s_{0,n}+\frac{x}{\sqrt{\a_n}}\Bigr)-E_{0,n}\Bigl(\s_{0,n}+\frac{x}{\sqrt{\a_n}}\Bigr)\Bigr]\\
&\qquad\qquad\qquad\qquad\qquad\qquad\le 
-\frac{1}{\sqrt{\a_n}}\Bigl[E_{0,n}\Bigl(\s_{0,n}+\frac{x}{\sqrt{\a_n}}\Bigr)-E_{0,n}(\s_{0,n}) \Bigr]\biggr).
\end{align*}
The variables in the left side of the last probability are asymptotically normal by
Lemma~\ref{LemAN}, while by the mean value theorem
the numbers on the right side of the inequality are equal to $-xE_{0,n}'(\s_n)/\a_n$ for numbers $\s_n$ between
$\s_{0,n}$ and $\s_{0,n}+x/\sqrt{\a_n}$ and hence  $\s_n\ra \s_0$.
Thus $-xE_{0,n}'(\s_n)/\a_n\ra x\tau_2^2$, by Lemma~\ref{LemmaZero}. The asymptotic normality of
$\hat\s_n$ follows.

\subsection{Proof of Theorem~\ref{ThmSigmaPost}}
\label{SectionProofThmSigmaPost}
We first prove that the posterior distribution is $\sqrt{\a_n}$-consistent:
$\Pi_n\bigl (\sqrt{\a_n}|\s-\s_{0,n}|>m_n\given X_1,\ldots, X_n\bigr)\ra 0$ in probability for any $m_n\ra\infty$.
By the monotonicity of $\Lambda_n'$ and the fact that $\Lambda_n'(\hat\s_n)=0$, for given $\s_n>\hat\s_n$, 
\begin{alignat*}{3}\
\Lambda_n(\s)&\ge\Lambda_n(\s_n),&\qquad &\text{ if } \hat\s_n<\s<\s_n,\\
\Lambda_n(\s)&\le\Lambda_n(\s_n)+\Lambda_n'(\s_n)(\s-\s_n),&\qquad &\text{ if } \s>\s_n.
\end{alignat*}
It follows that
\begin{align*}
\Pi_n\bigl (\s>\s_n\given X_1,\ldots, X_n\bigr)
&=\frac{\int_{\s_n}^1 e^{\Lambda_n(\s)}\,d\Pi_\s(\s)}{\int_0^1 e^{\Lambda_n(\s)}\,d\Pi_\s(\s)}\\
&\le \frac{\int_{\s_n}^1 e^{\Lambda_n(\s_n)+\Lambda_n'(\s_n)(\s-\s_n)}\,d\Pi_\s(\s)}{\int_{\hat\s_n}^{\s_n} e^{\Lambda_n(\s_n)}\,d\Pi_\s(\s)}\\
&\lesssim \frac{\int_0^\infty e^{\Lambda_n'(\s_n)u}\,du}{\s_n-\hat\s_n}=\frac1{-\Lambda_n'(\s_n)(\s_n-\hat\s_n)},
\end{align*}
where the proportionality constant depends on the density of $\Pi_\s$ only. If we choose $\s_n=\s_{0,n}+x/\sqrt{\a_n}$, then,
by Lemmas~\ref{LemAN} and~\ref{LemConvergenceExpectations}, since $E_{0,n}(\s_{0,n})=0$,
$$\frac{\Lambda_n'(\s_n)}{\sqrt{\a_n}}=O_P(1)+\frac{E_{0,n}(\s_{n})}{\sqrt{\a_n}}
=O_P(1)+ \frac{E_{0,n}'(\tilde\s_{n})}{\a_n}x=O_P(1)-\tau_2^2x.$$
Theorem~\ref{ThmSigma} gives that $\sqrt{\a_n}(\s_n-\hat\s_n)=x+O_p(1)$,  and 
hence the probability that $-\Lambda_n'(\s_n)(\s_n-\hat\s_n)$ is bigger
than some fixed constant can be made arbitrarily large by choosing large enough $x$. This shows that the
preceding display tends to zero in probability for $\s_n=\s_{0,n}+m_n/\sqrt{\a_n}$ and any $m_n\ra\infty$.
The probability of the events $\s_n>\hat\s_n$, on which the preceding displays are valid, 
then also tends to one.
Combined with a similar argument on the left tail of the posterior distribution, this shows that 
the posterior contracts to $\s_{0,n}$ at rate $1/\sqrt{\a_n}$.

Since $\sqrt{\a_n}(\hat\s_n-\s_{0,n})=O_p(1)$, by Theorem~\ref{ThmSigma},
there exists $m_n\ra\infty$ so that the sets $C_n:=\{\s: \sqrt{\a_n}|\s-\hat\s_n|\le m_n\}$
have posterior probability tending to one. The total variation measure
between the posterior measure $\Pi_n(\s\in \cdot\given X_1,\ldots, X_n)$ and the
conditioned posterior measure $\Pi_n(\s\in \cdot\given X_1,\ldots, X_n,\s\in C_n)$ is bounded
above by $2\Pi_n(\s\notin C_n\given X_1,\ldots, X_n)$ and  hence tends to zero in probability.
Thus it suffices to prove the Gaussian approximation to the conditioned posterior measure.

By Lemma~\ref{LemGCLambdpp} and the fact that $\Lambda_n'(\hat\s_n)=0$, a second order Taylor expansion gives
$$\sup_{\s\in C_n} \Bigl|\frac{\Lambda_n(\s)-\Lambda_n(\hat\s_n)}{\a_n(\s-\hat\s_n)^2}+\frac12 \tau_2^2\Bigr|\prob 0.$$
We conclude that there exist random variables  $\e_n$ that tend to zero in probability with,
for every $\s\in C_n$, 
$$-\frac12(\s-\hat\s_n)^2\a_n(\tau_2^2+\e_n)
\le \Lambda_n(\s)-\Lambda_n(\hat\s_n)
\le -\frac12(\s-\hat\s_n)^2\a_n(\tau_2^2-\e_n).$$
Then, for $\pi_\s$ a density of the prior measure $\Pi_\s$,
\begin{align*}
\Pi_n\bigl (\s\in B\given X_1,\ldots, X_n,\s\in C_n\bigr)
&\le \frac{\int_{B\cap C_n} e^{-(\s-\hat\s_n)^2\a_n(\tau_2^2-\e_n)/2}\,d\s}
{\int_{C_n} e^{-(\s-\hat\s_n)^2\a_n(\tau_2^2+\e_n)/2}\,d\s}
\frac{\displaystyle{\sup_{\s\in C_n}}\pi_\s(\s)}{\displaystyle{\inf_{\s\in C_n}}\pi_\s(\s)}
\end{align*}
By changing variables we see that $\Pi_n\bigl(\sqrt{\a_n}(\s-\hat\s_n)\in B\given X_1,\ldots, X_n,\s\in C_n)$
can be bounded above by
\begin{align*}
&\frac{\int_{-m_n}^{m_n}1_B(s) e^{-s^2(\tau_2^2-\e_n)/2}\,ds}{\int_{-m_n}^{m_n}e^{-s^2(\tau_2^2+\e_n)/2}\,ds}(1+o_P(1))\\
&\qquad\qquad= \frac{\Pr\bigl(Z/\sqrt{\tau_2^2-\e_n}\in B\cap (-m_n,m_n)\bigr)}{\Pr\bigl(Z/\sqrt{\tau_2^2+\e_n}\in (-m_n,m_n)\bigr)}(1+o_P(1)),
\end{align*}
for $Z$ a standard normal variable and the probabilities understood to refer to $Z$ only. This tends in
probability to $\Pr(Z/\tau_2\in B)$, uniformly in $B$.

By the same method, switching $+$ and  $-$ signs and $\sup$ and $\inf$, we can derive the same expression as an
asymptotic lower bound.  This concludes the proof of the first assertion of Theorem~\ref{ThmSigmaPost}.

Because convergence in total variation norm implies convergence of the expectations of bounded, measurable functions,
to prove the convergence of the posterior mean, it suffices to show that 
$$\int_{ \sqrt{\a_n}(\s-\hat\s_n)>m_n}\sqrt{\a_n}|\s-\hat\s_n|\,d\Pi_n(\s\given X_1,\ldots, X_n)\ra 0,$$ 
in probability, for every $m_n\ra\infty$, combined with a similar estimate on the left tail. 
By the argument at the beginning of the proof this expectation is bounded above by, for $\s_n=\hat\s_n+x/\sqrt{\a_n}$ with
$x>0$,
\begin{align*}
&\frac{\int_{\sqrt{\a_n}(\s-\hat\s_n)>m_n} \sqrt{\a_n}(\s-\hat\s_n)e^{\Lambda_n'(\s_n)(\s-\s_n)}\,d\Pi_\s(\s)}{\s_n-\hat\s_n}\\
&\qquad\qquad\qquad\qquad= \frac{\int_{m_n-x}^\infty(u+x)e^{\Lambda_n'(\s_n)u/\sqrt{\a_n}}\,du}{x}
\end{align*}
Because $\Lambda_n'(\s_n)/\sqrt {\a_n}=\sqrt{\a_n}(\s_n-\s_{0,n})\Lambda_n''(\tilde\s_n)/\a_n=(x+O_P(1))(-\tau_2^2+o_P(1))$,
the exponential in the integrand is with probability arbitrarily close to 1 bounded above by $e^{-cu}$, for some $c>0$, if
$x$ is chosen large enough. On the event  where this is the case the integral is bounded above by a multiple of 
$\int_{m_n-x}^\infty (u+x)e^{-cu}\,du\ra 0$, as $n\ra\infty$. 
It follows that the right side tends to zero in probability.


\subsection{Proof of Theorem~\ref{ThmM}}
\label{SectionProofThmM}
The maximum likelihood estimator of $M$ can be obtained in two steps:
first we maximize the log likelihood \eqref{EqLogLikelihood} over $\s$ for fixed $M$, yielding $\hat\s_{n,M}$, and next we 
maximize the ``profile log likelihood'' $M\mapsto \Lambda_n(\hat\s_{n,M}, M)$.
From Theorem~\ref{ThmSigma} we know that $\hat\s_{n,M}$ will be contained in a neighbourhood of
$\s_0$, with probability tending to 1. 
To proceed further,  we first  obtain a stochastic expansion of $\hat\s_{n,M}$ that refines this result.

The estimator $\hat\s_{n,M}$ solves $\Lambda_n'(\s,M)=0$, where the prime means the partial derivative with respect
to $\s$. The derivative can be written
\begin{align*}
\Lambda_n'(\s,M)&=\sum_{i=1}^{K_n-1}\frac i{M+i\s}-\sum_{l=1}^{n-1}\frac1{l-\s}Z_{n,l+1}
&=\frac{K_n}{\s}-G_n(\s)-\frac{h_{\s,M}(K_n)}{\s},
\end{align*}
for $G_n(\s)=\sum_{l=1}^{n-1}\frac1{l-\s}Z_{n,l+1}$, and 
\begin{equation}
\label{EqDefhsM}
h_{\s,M}(k)=1+\sum_{l=1}^{k-1}\frac M{M+l\s}\le 1+\frac M\s\log\Bigl(1+\frac {k\s}M\Bigr).
\end{equation}
Expansion of the equation $\Lambda_n'(\hat\s_{n,M},M)=0$ around the zero $\s_{0,n}$ of the function $E_{0,n}$ in \eqref{EqDefEn} gives
$$\hat \s_{n,M}-\s_{0,n}=-\frac{\Lambda'_n(\s_{0,n},M)}{\Lambda''_n(\s_{0,n},M)+\Lambda'''_n(\tilde \s_{n,M},M)(\hat\s_{n,M}-\s_{0,n})/2}.$$
It is shown in Lemma~\ref{LemAN} that 
$V_n:=\bigl(K_n/\s_{n,0}-G_n(\s_{n,0})\bigr)/\sqrt{\a_n}$, which is free of $M$, tends to a centered normal distribution, and
it is shown in Lemma~\ref{LemGCLambdpp} that the sequence $\bigl(-K_n/\s_{0,n}^2-G_n'(\s_{0,n})\bigr)/\a_n$ tends in probability
to $-\tau_2^2:=E_0'(\s_0)$. It can similarly be seen that $Z_n:= \bigl(2K_n/\s_{0,n}^3-G_n''(\s_{0,n})\bigr)/\a_n$ tends in probability to
a constant  $z$. Furthermore, it follows from the bound in \eqref{EqDefhsM} and the fact that $K_n/\a_n$ converges almost surely,
that $h_{\s_{0,n},M}(K_n)=O_P(\log n)$, uniformly in $M$ belonging to compacta, 
and from the definition of $h_{\s,M}(K_n)$ that its first and second partial derivatives relative to $\s$
are of the same order. Therefore, uniformly in $M$,
\begin{align*}
\Lambda'_n(\s_{0,n},M)&=\frac{K_n}{\s_{0,n}}-G_n(\s_{0,n})-\frac{h_{\s_{0,n},M}(K_n)}{\s_{0,n}}=\a_n^{1/2}V_n-\frac{h_{\s_{0,n},M}(K_n)}{\s_{0,n}},\\
\Lambda''_n(\s_{0,n},M)&=\!\frac{-K_n}{\s_{0,n}^2}-\!G_n'(\s_{0,n})-\!\frac d{d\s}\Bigl[\frac{h_{\s,M}(K_n)}{\s}\Bigr]_{\s=\s_{0,n}}\!\!\!\!\!
=\!-\a_n\tau_2^2+ \!O_P(\log n),\\
\Lambda'''_n(\tilde \s_{n,M},M)&=\!\frac{2K_n}{\s_{0,n}^3}-\!G_n''(\s_{0,n})-\!\frac {d^2}{d\s^2}\Bigl[\frac{h_{\s,M}(K_n)}{\s}\Bigr]_{\s=\s_{0,n}}\!\!\!\!\!
=\!\a_nZ_n+ O_P(\log n).
\end{align*}
Substituting these expansions in the preceding display gives that
\begin{align*}
\hat\s_{n,M}-\s_{0,n}&=\frac{\a_n^{-1/2}V_n-\a_n^{-1} h_{\s_{0,n},M}(K_n)/\s_{0,n}}{\tau_2^2-Z_n(\hat\s_{n,M}-\s_{0,n})+O_P(\log n/\a_n)}\\
&=\!\Bigl(\frac{V_n}{\a_n^{1/2}\tau_2^2}- \frac{h_{\s_{0,n},M}(K_n)}{\a_n\s_{0,n}\tau_2^2}\Bigr)
\Bigl(1+\frac{Z_n(\hat\s_{n,M}-\s_{0,n})}{\tau_2^2}+ O_P\Bigl(\frac{\log n}{\a_n}\Bigr)\Bigr).
\end{align*}
We can solve $\hat\s_{n,M}-\s_{0,n}$ from this as 
$$(\hat\s_{n,M}-\s_{0,n})\Bigl(1-\frac{V_nZ_n}{\a_n^{1/2}\tau_2^4}\Bigr)=
\frac{V_n}{\a_n^{1/2}\tau_2^2}-\frac{h_{\s_{0,n},M}(K_n)}{\a_n\s_{0,n}\tau_2^2}+O_P\Bigl(\frac{\log n}{\a_n^{3/2}}\Bigr).$$
Hence, for $W_{n,M}=h_{\s_{0,n},M}(K_n)/(\s_{0,n}\tau_2^2\log n)$,
\begin{align*}
\hat\s_{n,M}&=\s_{0,n}+\frac{V_n}{\a_n^{1/2}\tau_2^2}+\frac{V_n^2Z_n}{\a_n\tau_2^6}-\frac{W_{n,M}\log n} {\a_n}+O_P\Bigl(\frac{\log n}{\a_n^{3/2}}\Bigr)\\
&=:\tilde\s_n-W_{n,M}\log n/\a_n+O_P(\log n/\a_n^{3/2}).
\end{align*}
The variables $V_n$, $Z_n$ and $W_{n,M}$ are all bounded in probability.
The quantity $\tilde\s_n$ is defined as the sum of the first three terms on the
right in the preceding line, and does not depend on $M$.

We are now ready to expand the profile likelihood $\Lambda_n(\hat\s_{n,M},M)$. 
The first and third terms on the far right side of \eqref{EqLogLikelihood} can
be expanded with the help of Lemma~\ref{LemmaSum} as, 
uniformly in $M$ in compacta,
\begin{align*}
\sum_{l=1}^{K_n-1}\log (M+l\hat\s_{n,M})&=K_n\log K_n+K_n\log (\hat \s_{n,M}/e)+\Bigl(\frac{M}{\hat\s_{n,M}}-\frac 12\Bigr)\log K_n\\
&\qquad+\log \frac{\sqrt{2\pi}}{\hat\s_{n,M}}-\log \Gamma\Bigl(1+\frac{M}{\hat\s_{n,M}}\Bigr)+O_P\Bigl(\frac{1}{K_n}\Bigr),\\
\sum_{i=1}^{n-1} \log (M+i)&=n\log n-n+\Bigl(M-\frac12\Bigr)\log n+\\
&\qquad+\log \sqrt{2\pi}-\log \Gamma(1+M)+O\Bigl(\frac{1}{n}\Bigr).
\end{align*}
To find the point of maximum $M$, we can drop terms that depend on $K_n$ and $n$ only, and
add terms that do not depend on $M$. Thus finding the maximizer of the profile likelihood is 
equivalent to finding the maximizer of 
\begin{align}
&(K_n-1)\log \frac{\hat \s_{n,M}}{\tilde\s_n}+\sum_{l=1}^{n-1}\log \frac{l-\hat\s_{n,M}}{l-\tilde \s_n}Z_{n,l+1}
+\frac{M\log K_n}{\hat\s_{n,M}}-M\log n\nonumber\\
&\qquad\qquad \qquad+\log \Gamma(1+M)-\log \Gamma\Bigl(1+\frac{M}{\hat\s_{n,M}}\Bigr)+O_P\Bigl(\frac1{K_n}\Bigr).
\label{EqTobeMaxed}
\end{align}
The sum of the first two terms can be expanded as
\begin{align*}
&(K_n-1)\log \Bigl(1-\frac{W_{n,M}\log n}{\a_n\tilde\s_n}+O_P\Bigl(\frac{\log n}{\a_n^{3/2}}\Bigr)\Bigr)
-\sum_{l=1}^{n-1}\log\Bigl(1-\frac{ \hat\s_{n,M}-\tilde\s_n}{l-\tilde\s_n}\Bigr)Z_{n,l+1}\\
&\qquad=-\frac{K_nW_{n,M}\log n}{\a_n\tilde\s_n}
-\sum_{l=1}^{n-1}\Bigl(\frac{ \hat\s_{n,M}-\tilde\s_n}{l-\tilde\s_n}\Bigr)Z_{n,l+1}+O_P\Bigl(\frac{\log n}{\sqrt \a_n}\Bigr)\\
&\qquad=-\frac{W_{n,M}\log n}{\a_n}\Bigl(\frac{K_n}{\tilde\s_n}- \sum_{l=1}^{n-1}\frac{Z_{n,l+1}}{l-\tilde\s_n}\Bigr)
+O_P\Bigl(\frac{\log n}{\sqrt \a_n}\Bigr)\\
&\qquad=-\frac{W_{n,M}\log n}{\a_n}\bigl(\sqrt{\a_n}\tilde V_n+E_{0,n}(\tilde\s_n)\bigr)+O_P\Bigl(\frac{\log n}{\sqrt{\a_n}}\Bigr),
\end{align*}
where $\tilde V_n$ tends to a centered normal distribution by Lemma~\ref{LemAN}. The right side
is of the order $O_P(\log n/\sqrt {\a_n})$.

Since $L_n=\sqrt{\a_n}\bigl(K_n/\a_n-\Gamma(1-\s_0)\bigr)$ tends to a centered normal
distribution, we have $K_n=\a_n\Gamma(1-\s_0)+O_P(\sqrt{\a_n})$, so that 
$\log K_n=\log \a_n+\log\Gamma(1-\s_0)+O_P(1/\sqrt{\a_n})$.
Therefore, if $\a_n=n^{\s_0}\bar L_0(n)$, then the sum of the third and fourth terms on the
right of \eqref{EqTobeMaxed} are
\begin{align*}
&M\!\Bigl[\frac{\s_0\log n\!+\!\log \bar L_0(n)\!+\!\log \Gamma(1-\s_0)\!+\!O_P(1/\!\sqrt{\a_n})}{\s_{0,n}}\Bigr]\!
\Bigl[\!1\!+\! O_P\Bigl(\!\frac{1}{\sqrt{\a_n}}\!\Bigr)\!\Bigr]\!-\!M\log n\\
&\qquad=M\Bigl(\frac{\s_0}{\s_{0,n}}-1\Bigr)\log n 
+\frac M{\s_{0,n}}\bigl(\log \bar L_0(n)+\log\Gamma(1-\s_0)\bigr)+O_P\Bigl(\frac{\log n}{\sqrt{\a_n}}\Bigr).
\end{align*}
By Lemma~\ref{LemmaZero}, $\s_{0,n}-\s_0=O(n^{-\d})+O\bigl(L_0'(n)n/L_0(n)\bigr)$,
whence the right side is of the order $(\log n) L_0'(n)n/L_0(n)+\log L_0(n)+\log \Gamma(1-\s_0)$.

We conclude that up to terms that do not depend on $M$, the profile log likelihood \eqref{EqTobeMaxed}
is equal to, for a constant $c>0$,
\begin{align*}
&\frac{M}{\s_0}\Bigl[\frac{L_0'(n)nc\log n}{L_0(n)}+\log L_0(n)+\log \Gamma(1-\s_0)\Bigr]\\
&\qquad\qquad\qquad\qquad\qquad+\log \Gamma(1+M)-\log \Gamma\Bigl(1+\frac{M}{\s_0}\Bigr)+o_P(1).
\end{align*}
If the term within square brackets tends to infinity,  then asymptotically this term dominates and
the maximizing value $\hat M_n$ will tend to the end point of the interval. If this term tends to minus
infinity, then it also dominates and the maximum value will tend to 0. If the term converges to a limit,
then the whole expression tends to a function of the form 
$M\mapsto aM+\log \Gamma(1+M)-\log \Gamma(1+{M}/{\s_0})$. This function is concave and tends
to $-\infty$ as $M\ra\infty$. Its derivative at zero may be both positive or negative, depending on $a$ and $\s_0$,
and the point of maximum of the function may be at zero or at some positive location, possibly the upper limit
$\bar M$ of the parameter. The sequence $\hat M_n$ will tend to this point of maximum.

\subsection{Lemmas}
For $g_\s(m)=\sum_{l=1}^{m-1}\frac {1}{l-\s}$, set
\begin{align}
\label{EqDeftau1Echt}
\tau_1^2&=\frac{(2^{\s_0}-1)\Gamma(1-\s_0)}{\s_0^2}
+\sum_{m=1}^\infty \frac{\Gamma(m-\s_0)(g_{\s_0}(m+1)+g_{\s_0}(m))}{m!}\\
&\qquad-\sum_{k=2}^\infty\sum_{m=1}^\infty\frac{g_{\s_0}(k)\Gamma(k+m+1-\s_0)}{k!m!2^{k+m-\s_0}(m-\s_0)}
 -\sum_{m=2}^\infty \frac{g_{\s_0}(m) \Gamma(m-\s_0)}{m!\,2^{m-\s_0-1}}.\nonumber
\end{align}

\begin{lemma}
\label{LemAN}
For any $\s_n\prob\s_0\in (0,1)$ and $M_n=o\bigl(\sqrt{\a_0(n)}/\log n\bigr)$, we have
$\a_0(n)^{-1/2}\bigl(\Lambda_n'(\s_n,M_n)-E_{0,n}(\s_n)\bigr)\weak N(0,\tau_1^2)$.
\end{lemma}

\begin{proof}
We denote the true distribution by $P_0 = \sum_{i = 1}^\infty p_i \delta_{x_i}$.
The variables $Z_{n,l}$ can be written as  $Z_{n,l}=\sum_{j=1}^\infty1_{M_{n,j}\ge l}$,
for  $M_{n,j}$  the number of observations equal to $x_j$.  As $K_n=Z_{n,1}$, the
 function $\Lambda_n'$ can be written in the form
$$\Lambda_n'(\s,M)
=\sum_{j=1}^\infty \Bigl[\frac{1_{M_{n,j}\ge 1}}\s-g_\s(M_{n,j})\Bigr]-\frac{h_{\s,M}(K_n)}{\s},$$
where $g_\s(0)=g_\s(1)=0$ and $g_\s(m)=\sum_{l=1}^{m-1}\frac {1}{l-\s}$, for $m\ge 2$.
It is shown in \cite{Karlin1967} (and repeated below) that $\E K_n/\a_0(n)\ra \Gamma(1-\s_0)$ and hence
Jensen's inequality and \eqref{EqDefhsM} give
$\E h_{\s_n,M}(K_n)\le 1+(M/\s_n)\log (1+\E K_n \s_n/M)=O(M\log n)=o(\a_0(n)^{1/2})$, 
so that the term on the far right is asymptotically negligible.

To prove the asymptotic normality of the infinite sum, after centering and scaling, we first consider the
case that the sample size $n$ is taken to be a random variable $N_n$ with the Poisson distribution with mean $n$, independent
of the original variables. The resulting variables $M_{N_n,j}$ are then Poisson distributed with means $np_j$ 
and independent across $j$, and the asymptotic normality can be proved using the Lindeberg central limit theorem.
(For reference, an appropriate infinite series version is formulated in Lemma~\ref{LemmaLindeberg}.)
As shown in the proof of (i) and (iii) of Lemma~\ref{LemConvergenceExpectations}, 
the function $E_{0,n}(\s)$ defined in \eqref{EqDefEn}
is equal to the expectation of $V_{k,n}=\sum_{j=1}^\infty\bigl[1_{M_{n,j}\ge 1}/\s-g_\s(M_{n,j})\bigr]$.
Because $(1_{M_{n,j}=0})g_\s(M_{n,j})=0$, by  Lemma~\ref{LemConvergenceExpectations} (ii), (v), (iv) and (viii), 
the variance of $V_{k,n}$ is given by 
\begin{align}
&\sum_{j=1}^\infty \var \Bigl(\frac{1_{M_{N_n,j}\ge 1}}\s-g_\s(M_{N_n,j})\Bigr)
=\sum_{j=1}^\infty \Bigl[\frac{e^{-np_j}(1-e^{-np_j})}{\s^2}\nonumber\\
&\qquad\qquad+\var g_\s(M_{N_n,j})-2\frac{e^{-np_j}}{\s}\E g_\s(M_{N_n,j})\Bigr]\sim \a_0(n)\tau_1^2,
\label{EqDeftau1}
\end{align}
where $\tau_1^2$ is equal to $(ii)/\s^2+(v)-(vi)-(2/\s)(vii)$, for $(ii), (v), (vi)$ and $(vii)$ shorthand for
the expressions on the right sides in Lemma~\ref{LemConvergenceExpectations} (ii), (v), (vi) and (vii).
Furthermore, the variables $1_{M_{N_n,j}\ge 1}$ are bounded and hence trivially satisfy the
Lindeberg condition, while the Lindeberg condition on the variables $g_\s(M_{N_n,j})$ follows from
the boundedness of $\a_0(n)^{-1}\sum_j\E g_\s(M_{N_n,j})^3$, by Lemma~\ref{LemConvergenceExpectations} (viii).
Thus the Poissonized sums are asymptotically normal, by Lemma~\ref{LemmaLindeberg}.

The proof can be completed by  applying Lemma~\ref{LemmaPoissonization}  to the variables
$V_{k,n}=\sum_{j=1}^\infty \bigl(1_{M_{k,j}\ge 1}/\s_n-g_{\s_n}(M_{k,j})\bigr)$, with $a_n=\a_0(n)^{-1/2}$ and 
$E_{0,n}$ as given. To show that $a_n(V_{k_n,n}-V_{n,n})$ tends to zero in probability, we split $V_{k,n}$ in 
$\sum_j 1_{M_{k,j}\ge1}/\s_n$ and $\sum_j g_{\s_n}(M_{k,j})$ and handle the two parts separately.
Because $a_{k_n}/a_n\ra 1$, it is not a loss of generality to assume that $k_n\ge n$. 
Because the binomial distributions $\text{binomial}(n,p_j)$
are stochastically ordered in $n$ and the functions $m\mapsto 1_{m\ge 1}$ and $m\mapsto g_\s(m)$ 
are increasing, the variables $a_n\sum_j (1_{M_{k,j}\ge1}- 1_{M_{n,j}\ge1})/\s$ and $a_n\sum_j (g_{\s_n}(M_{k,j})- g_{\s_n}(M_{n,j}))$
are nonnegative, and hence it suffices to show that their expectations tend to zero. 
This is shown in Lemma~\ref{LemmaInvariancenkn}.
\end{proof}

\begin{lemma}
\label{LemGCLambdpp}
For any $\tilde\s_n\prob \s_0\in(0,1)$ and $M_n=o\bigl(\a_0(n)/\log n\bigr)$, we have
 $\a_0(n)^{-1}\Lambda_n''(\tilde\s_n,M_n)\ra E_0'(\s_0)$,
in probability.
\end{lemma}

\begin{proof}
The second derivative is given by
\begin{align*}\Lambda_n''(\s)&=-\sum_{l=1}^{K_n-1}\frac{l^2}{(M+l\s)^2}-\sum_{l=1}^{n-1}\frac1{(l-\s)^2}Z_{n,l+1}\\
&=-\frac{K_n-1}{\s^2}+\frac 1{\s^2}\sum_{l=1}^{K_n-1}\Bigl[\frac{2M}{M+\s l}-\frac{M^2}{(M+\s l)^2}\Bigr]
-\sum_{l=1}^{n-1}\frac1{(l-\s)^2}Z_{n,l+1}.
\end{align*}
It is shown in \cite{Karlin1967} (see his formula (66), or see the proof of Lemma~\ref{LemAN}), that
$K_n/\a_0(n)\ra \Gamma(1-\s_0)$ and $Z_{n,l}/\a_0(n)\ra \Gamma(l-\s_0)/(l-1)!$, for every $l\ge 1$, in probability and in mean.
We use this to infer the convergence of the first and last terms on the right divided by $\a_0(n)$. That the limit is equal
to $E_0'(\s_0)$ follows by inspection of its form and Lemma~\ref{LemmaZero}. The second 
term is bounded above in absolute value  by a multiple of $M\log K_n$ and divided  by $\a_0(n)$ tends to zero.
\end{proof}

\subsection{Technical lemmas}
In the next lemmas $(p_j)_{j=1}^\infty$ is a given infinite probability vector and
$\a$ is the cumulative distribution function of the counting measure on the points $1/p_j$, for $j\in \NN_+$.
Furthermore, the function $g_\s: \NN\to\NN$ is given by $g_\s(1)=g_\s(2)=0$ and
$$g_\s(m)=\sum_{l=1}^{m-1}\frac {1}{l-\s}, \qquad m\ge 2.$$

\begin{lemma}
\label{LemConvergenceExpectations}
Suppose that $\a(u):=\#\{ j: 1/p_j\le u\}$ is regularly varying at $\infty$ of order $\g\in (0,1)$.
Then, for any $\s_n\ra\s\in (0,1)$, and independent $M_{n,j}\sim\text{Poisson}(np_j)$,
\begin{itemize}
\item[(i)] $\frac{1}{\a(n)}\sum_{j=1}^\infty\E 1_{M_{n,j}\ge 1}\ra\Gamma(1-\g)$,
\item[(ii)] $\frac{1}{\a(n)}\sum_{j=1}^\infty\var 1_{M_{n,j}\ge 1}\ra(2^{\g}-1)\Gamma(1-\g)$,
\item[(iii)] $\frac{1}{\a(n)}\sum_{j=1}^\infty\E g_{\s_n}(M_{n,j})\ra \sum_{m=1}^\infty \frac{\Gamma(m+1-\g)}{m!(m-\s)}$,
\item[(iv)] $\frac{1}{\a(n)}\sum_{j=1}^\infty\E \frac{\partial g_{\s_n}}{\partial \s}(M_{n,j})
\ra \sum_{m=1}^\infty \frac{\Gamma(m+1-\g)}{m!(m-\s)^2}$,
\item[(v)] $\frac{1}{\a(n)}\sum_{j=1}^\infty\E g_{\s_n}^2(M_{n,j})
\ra\sum_{m=1}^\infty \frac{\Gamma(m+1-\g)(g_\s(m+1)+g_\s(m))}{m!(m-\s) }$,
\item[(vi)] $\frac{1}{\a(n)}\sum_{j=1}^\infty\bigl(\E g_{\s_n}(M_{n,j})\bigr)^2
\ra \sum_{k=2}^\infty\sum_{m=1}^\infty\frac{g_\s(k)\Gamma(k+m+1-\g)}{k!m!2^{k+m-\g}(m-\s)}$.
\item[(vii)] $\frac{1}{\a(n)}\sum_{j=1}^\infty e^{-np_j} \E g_{\s_n}(M_{n,j})\ra \sum_{m=2}^\infty \frac{g_\s(m)\g \Gamma(m-\g)}{m!\,2^{m-\g}}$.
\item[(viii)] $\frac{1}{\a(n)}\sum_{j=1}^\infty\E g_{\s_n}^3(M_{n,j})
\ra\sum_{m=1}^\infty \frac{\Gamma(m+1-\g)(g_\s^2(m+1)+g_\s(m+1)g_\s(m)+g_\s^2(m))}{m!(m-\s) }$,
\end{itemize}
All limits on the right sides are finite.
\end{lemma}

\begin{proof}
Assertions (i) and (ii) were stated in \cite{Karlin1967}; we include proofs for completeness.

The series in the left side of (i) is
$$\sum_{j=1}^\infty \Pr(M_{n,j}\ge 1)=\sum_{j=1}^\infty (1-e^{-np_j})=\int_1^\infty\!(1-e^{-n/u})\,d\a(u)
=\int_0^n\!\a\Bigl(\frac ns\Bigr)e^{-s}\,ds,$$
by Fubini's theorem (or partial integration), since $1-e^{-n/u}=\int_0^{n/u} e^{-s}\,ds$. By the definition of
regular variation $\a(n/s)/\a(n)\ra s^{-\g}$, for every $s$, as $n\ra\infty$. By Potter's theorem (\cite{Binghametal}, Theorem 1.5.6),
for every $\d>0$ there exists $M>1$ such that $\a(n/s)/\a(n)\le s^{-\g-\d}\vee s^{-\g+\d}$, for every $s<n/M$. We can choose $\d$ so that 
$\g+\d<1$, and then $\int_0^\infty (s^{-\g-\d}\vee s^{-\g+\d})e^{-s}\,ds<\infty$. For the corresponding $M$, we then have $\int_0^{n/M}\a(n/s)/\a(n)\,e^{-s}\,ds
\ra\int_0^\infty s^{-\g}e^{-s}\,ds=\Gamma(1-\g)$, by the dominated convergence theorem. For $s\ge n/M$, we have $\a(n/s)\le \a(M)$
and hence $\int_{n/M}^n\a(n/s) e^{-s}\,ds\le \a(M)e^{-n/M}=o(\a(n))$, as $n\ra\infty$. 

The series in (ii) is $\sum_{j=1}^\infty e^{-np_j}(1-e^{-np_j})=\sum_{j=1}^\infty (1-e^{-2np_j})-\sum_{j=1}^\infty (1-e^{-np_j})$.
By the first paragraph this is asymptotic to $\bigl(\a(2n)-\a(n)\bigr)\Gamma(1-\g)\sim (2^\g-1)\a(n)\Gamma(1-\g)$, by regular
variation of $\a$.

For (iii) we write 
$$\sum_{j=1}^\infty\E g_{\s}(M_{n,j})=\!\sum_{j=1}^\infty\sum_{m=2}^\infty g_\s(m)\frac{e^{-np_j}(np_j)^m}{m!}\!
=\sum_{m=2}^\infty\frac{g_\s(m)}{m!}\!\int_1^\infty\!\! e^{-n/u}\Bigl(\frac nu\Bigr)^m\,d\a(u).$$
Substituting  $e^{-n/u}(n/u)^m=\int_0^{n/u} e^{-s}s^{m-1}(m-s)\,ds$ (valid for $m>0$) and using Fubini's theorem, we can rewrite the
right side as
\begin{align*}
&\sum_{m=2}^\infty\frac{g_\s(m)}{m!}\int_0^n \a\Bigl(\frac n s\Bigr) e^{-s} s^{m-1}(m-s)\,ds\\
&=g_\s (2)\int_0^n\a\Bigl(\frac n s\Bigr) e^{-s} s\,ds+\sum_{m=2}^\infty\frac{g_\s(m+1)-g_\s(m)}{m!}\int_0^n \a\Bigl(\frac n s\Bigr) e^{-s} s^{m}\,ds\\
&=\int_0^n\a\Bigl(\frac n s\Bigr) e^{-s}\Bigl(\sum_{m=1}^\infty\frac{s^m}{m!(m-\s)}\Bigr) \,ds.
\end{align*}
As before, regular variation and Potter's theorem give for $s\le n/M$ the bound $\a(n/s)/\a(n)\lesssim s^{-\g-\d}\vee s^{-\g+\d}$, and then 
$$\frac{\a(n/s)}{\a(n)} \Bigl(\sum_{m=1}^\infty\frac{s^m}{m!(m-\s)}\Bigr) \lesssim (s^{-\g-\d}\vee 1) (e^s-1-s)\frac 1s,\qquad s\le n/M.$$
Furthermore, the left side tends pointwise to $s^{-\g}\sum_{m=1}^\infty s^m/(m!(m-\s))$. By the dominated convergence theorem, 
\begin{align*}
&\int_0^{n/M}\frac{\a(n/ s)}{\a(n)} e^{-s}\Bigl(\sum_{m=1}^\infty\frac{s^m}{m!(m-\s)}\Bigr) \,ds
\ra \int_0^\infty s^{-\g}e^{-s}\sum_{m=1}^\infty \frac{s^m}{m!(m-\s)}\,ds.
\end{align*}
The right side is the limit as given.
Since $\sum_j p_j=1$, we have that $\a(u)=\# (p_j\ge 1/u)\le u$. Therefore 
$$\int_{n/M}^n\a\Bigl(\frac n s\Bigr) e^{-s}\Bigl(\sum_{m=1}^\infty\frac{s^m}{m!(m-\s)}\Bigr) \,ds
\le \int_{n/M}^\infty \frac n s \frac 1 s\,ds\le M.$$
This is of lower order than $\a(n)$ and hence the proof of the third assertion is complete.

For (iv) we follow the same approach as in (iii), replacing $g_\s$ by $\dot g_\s=\partial/\partial\s g_s$,
and then at the end substitute $\dot g_s(m+1)-\dot g_\s(m)=1/(m-\s)^2$.

For (v) again we follow the same approach as under (iii), now replacing $g_\s$ by $g_\s^2$. 
At the end we write the difference $g_\s^2(m+1)-g_\s^2(m)$ as 
$(m-\s)^{-1}\bigl(g_\s(m+1)+g_\s(m)\bigr)$ and complete the argument as before, where we
can bound $g_\s(m+1)+g_\s(m)$ by a multiple of $\log m$, for large $m$, and use that
$\sum_m s^m\log m/m!\lesssim e^s(s^\d\vee 1)$, for every $s$, by Lemma~\ref{LemmaMisc},
with a sufficiently small $\d>0$.

The series in (vi) is equal to
\begin{align*}
&\sum_{j=1}^\infty\sum_{k=2}^\infty\sum_{m=2}^\infty g_\s(k)g_\s(m)\frac{e^{-2np_j}(np_j)^{k+m}}{k!m!}\\
&\qquad\quad=\sum_{k=2}^\infty\sum_{m=2}^\infty\frac{g_\s(k)g_\s(m)}{k!\,m!}\!\int_1^\infty\!\! e^{-2n/u}\Bigl(\frac nu\Bigr)^{k+m}\,d\a(u).
\end{align*}
Substituting  $e^{-2n/u}(2n/u)^{k+m}=\int_0^{2n/u} e^{-s}s^{k+m-1}(k+m-s)\,ds$ and using Fubini's theorem, we can rewrite the
right side as
\begin{align*}
&\sum_{k=2}^\infty\sum_{m=2}^\infty\frac{g_\s(k)g_\s(m)}{k!\,m!\,2^{k+m}}\int_0^{2n} \a\Bigl(\frac{2n} s\Bigr) e^{-s} s^{k+m-1}(k+m-s)\,ds\\
&\quad=\int_0^{2n} \a\Bigl(\frac {2n} s\Bigr)\,\frac{d}{ds} \Bigl[\Bigl(\sum_{k=2}^\infty\frac{g_\s(k)s^{k}}{k!\,2^{k}}\Bigr)^2e^{-s}\Bigr]\,ds\\
&\quad=\int_0^{2n} \a\Bigl(\frac {2n} s\Bigr)\,\Bigl(\sum_{k=2}^\infty\frac{g_\s(k)s^{k}}{k!\,2^{k}}\Bigr)e^{-s}
\Bigl[\sum_{k=1}^\infty\frac{\bigl(g_\s(k+1)-g_\s(k)\bigr)s^k}{k!\,2^{k}}\Bigr]\,ds.
\end{align*}
In view of Lemma~\ref{LemmaMisc} and because $g_\s(k+1)-g_\s(k)=1/(k-\s)$, the integrand is bounded above by a multiple of 
$\a(2n/s) (e^{s/2}-1)e^{-s}s^{-1}(e^{s/2}-1)$. Using the dominated convergence theorem and
arguments as before, we see that the right side divided by $\a(n)$ is asymptotic to the right side of (vi).

The extra factor $e^{-np_j}$ in (vii) relative to (iii) leads to the same expression as in (iii), except
that $e^{-n/u}$ must be replaced by $e^{-2n/u}$. Following the same argument, we find that the series in (vii)
is equal to 
$$\sum_{m=2}^\infty \frac{g_\s(m)}{m!\, 2^m}\int_0^{2n}\a\Bigl(\frac {2n }s\Bigr)e^{-s}s^{m-1}(m-s)\,ds.$$
The functions $\sum_{m=2}^\infty \frac{g_\s(m)}{m! 2^m}e^{-s}s^{m-1}|m-s|$ are uniformly integrable
(thanks to the factors $2^m$ that are extra relative to (iii)) . Therefore, by arguments as before 
the display is asymptotically equivalent to the expression obtained by replacing $\a(2n/s)$ by $\a(n)(2/s)^{\g}$.
Finally we can use that  $m\Gamma(m-\g)-\Gamma(m+1-\g)=\g\Gamma(m-\g)$.

For (viii) we follow the same approach as under (iii), replacing $g_\s$ by $g_\s^3$, where at
the end we write the difference $g_\s^3(m+1)-g_\s^3(m)$ as $(m-\s)^{-1}(g_\s^2(m+1)+g_\s(m+1)g_\s(m)+g_\s^2(m))$.

The finiteness of the limitss can be proved  with the help of Lemma~\ref{LemmaStirling} by comparison with standard series.
\end{proof}

For $\a(u):=\#\{ j: 1/p_j\le u\}$ as in the preceding, define a function $E_n$ by 
\begin{equation}
E_n(\s)= \int_0^n\a\Bigl(\frac ns\Bigr)e^{-s}\Bigl(\frac 1\s-\sum_{m=1}^\infty \frac{s^m}{m!(m-\s)}\Bigr)\,ds.
\label{EqDefEnTwo}
\end{equation}

\begin{lemma}
\label{LemmaZero}
If the function $\a$ is regularly varying at $\infty$ of order $\g\in (0,1)$, then
the functions $E_n$ in \eqref{EqDefEnTwo} satisfy, as $n\ra\infty$, for $\s_n\ra\s\in(0,1)$,
\begin{align}\frac{E_n(\s_n)}{\a(n)}&\ra  E(\s):=\frac{\Gamma(1-\g)}\s-\sum_{m=1}^\infty\frac{\Gamma(m+1-\g)}{m!(m-\s)},
\label{EqDefETwo}\\
\frac{E_n'(\s_n)}{\a(n)}&\ra E'(\s).\nonumber
\end{align}
The limit function $E$ vanishes at $\s=\g$ and the zeros $\s_{0,n}$ of $E_n$ satisfy $\s_{0,n}\ra\g$.
Furthermore, if there exists a continuously differentiable function $L: [1,\infty)\to \RR$ such that
$|\a(u)-u^\g L(u)|\le C u^\b$, for every $u>1$, and some $C>0$ and $\b<\g$, and
such that $s\mapsto L'(s)s$ is slowly varying at $\infty$, then, as $n\ra\infty$,
$$\s_{0,n}-\g=O\Bigl(\frac{n^{\b-\g}}{L(n)}\Bigr)-\frac{\Gamma(1-\g)(1+\g)}{\g^2E '(\g)}\Bigl(\frac{L'(n)n}{L(n)}\Bigr)\bigl(1+o(1)\bigr).$$ 
In particular, if $L$ can be taken constant, then $\s_{0,n}-\g=O(n^{\b-\g})$.
\end{lemma}

\begin{proof}
As shown in the proof of Lemma~\ref{LemConvergenceExpectations} (i) and (iii), the function
$E_n(\s)$ is equal to $\sum_{j=1}^\infty\bigl( \E 1_{M_{n,j}\ge 1}/\s -\E g_\s(M_{n,j})\bigr)$, with the notation
as in the lemma. Therefore, by the lemma the limit function $E$ is equal to the limit in (i)  divided by $\s$ minus the limit in (iii),
i.e.\ the right side of \eqref{EqDefETwo}.
The limit  in (iii) at $\s=\g$ can be written
\begin{align*}
    \sum_{m=1}^\infty \frac{\Gamma(m-\g)}{m!}=\sum_{m=1}^\infty \int_0^\infty\frac{s^{m-\g-1}}{m!}e^{-s}\,ds
=\int_0^\infty(1-e^{-s})s^{-\g-1}\,ds.
\end{align*}
By partial integration, this can be further rewritten as $\int_0^\infty x^{-\g}/\g \,e^{-x}\,dx=\Gamma(1-\g)/\g$.
Thus $E(\g)=\Gamma(1-\g)/\g- \Gamma(1-\g)/\g=0$.

The limit  of $E_n'(\s_n)/\a(n)$ is obtained similarly from (i) and (iv) of 
Lemma~\ref{LemConvergenceExpectations}, and it is seen to be equal to the derivative $E'(\s)$.

The functions $E_n$ and $E$ are monotonely decreasing and $E_n(\g-\e)/\a(n)\ra E(\g-\e)>0$ and
$E_n(\g+\e)/\a(n)\ra E(\g+\e)<0$, for every $\e>0$, by \eqref{EqDefETwo}. This shows that $\s_{0,n}\in (\g-\e,\g+\e)$ eventually
and hence $\s_{0,n}\ra\g$.

By the mean value theorem, $E_n(\s_{0,n})-E_n(\g)=E_n'(\tilde\s_n)(\s_{0,n}-\g)$, for some $\tilde\s_n\ra\g$.
Since $E_n(\s_{0,n})=0$ and $E_n'(\tilde \s_n)/\a(n)\ra E'(\g)<0$, 
it follows that $\s_{0,n}-\g=-\bigl(E_n(\g)/\a(n)\bigr)/\bigl(E'(\g)+o(1)\bigr)$.
This shows that $\s_{0,n}\ra\g$ at the same rate of convergence as $E_n(\g)/\a(n)\ra E(\g)=0$. 

To investigate the latter rate, define functions $H_n$ by 
$$H_n(\s)=\int_0^n L\Bigl(\frac ns\Bigr)s^{-\g} e^{-s}\Bigl(\frac1\s-\sum_{m=1}^\infty\frac{s^m}{m!(m-\s)}\Bigr)\,ds.$$
The functions $H_n$ are monotonely decreasing, with, by the assumption on $\a$,
\begin{align*}
\bigl|E_n(\s)-n^\g H_n(\s)\bigr|
&\le \int_0^n \Bigl(\frac ns\Bigr)^\b e^{-s}\Bigl|\frac1\s-\sum_{m=1}^\infty\frac{s^m}{m!(m-\s)}\Bigr|\,ds\\
&\lesssim n^\b\Bigl(\frac 1\s+\frac1{1-\s}\int_0^n\frac1{s^{1+\b}}e^{-s}(e^s-1-s)\,ds\Bigr).
\end{align*}
The right side is $O(n^\b)$, if $\s$ is bounded away from 0 and 1.

Let $\bar H_n$ be defined as $H_n$, but with the integral extended from $(0,n)$ to the full half line $(0,\infty)$.
The assumption on $\a$ does not specify $L(u)$ for $u\le 1$, and hence does not specify $L(n/s)$ for $s\ge n$,
but we can extend the function $L$ to a continuously differentiable function on $(0,\infty)$ in such a way that it
vanishes on a neighbourhood of $0$ and hence is uniformly bounded on $[0,1]$.
Then $|H_n(\s)-\bar H_n(\s)|\lesssim \int_n^\infty s^{-\g} e^{-s}(1+s^{-1}(e^s-1))\,ds=O(n^{-\g})$, if
$\s$ is bounded away from 0 and 1.

Splitting the two parts of the integrand in $\bar H_n$ and  performing partial integration on the second part we find
\begin{align*}
\bar H_n(\g)&=\int_0^\infty L\Bigl(\frac ns\Bigr)\frac{s^{-\g}}\g e^{-s}\,ds 
-\int_0^\infty L\Bigl(\frac ns\Bigr)\sum_{m=1}^\infty\frac{s^{m-\g-1}}{m!}\,e^{-s}\,ds\\
&\qquad\qquad\qquad\qquad\qquad\qquad
+\int_0^\infty L'\Bigl(\frac ns\Bigr)\frac n{s^2}\sum_{m=1}^\infty\frac{s^{m-\g}}{m!(m-\g)}\,e^{-s}\,ds\\
&=\int_0^\infty L\Bigl(\frac ns\Bigr)\frac{s^{-\g}}\g e^{-s}\,ds 
+\int_0^\infty L\Bigl(\frac ns\Bigr)(1-e^{-s})\,d\Bigl(\frac{s^{-\g}}{\g}\Bigr)\\
&\qquad\qquad\qquad\qquad\qquad\qquad\qquad
+\int_0^\infty L'\Bigl(\frac ns\Bigr)\frac n{s^2}\sum_{m=1}^\infty\frac{s^{m-\g}}{m!(m-\g)}\,e^{-s}\,ds\\
&=0+\int_0^\infty L'\Bigl(\frac ns\Bigr)\frac n{s^2}\Bigl[(1-e^{-s})\frac{s^{-\g}}{\g}
+\sum_{m=1}^\infty\frac{s^{m-\g}e^{-s}}{m!(m-\g)}\Bigr]\,ds\\
&\sim L'(n)n\int_0^\infty \frac 1{s}\Bigl[(1-e^{-s})\frac{s^{-\g}}{\g}+\sum_{m=1}^\infty\frac{s^{m-\g}e^{-s}}{m!(m-\g)}\Bigr]\,ds,
\end{align*}
since the function $s\mapsto L'(s)s$ is slowly varying at infinity, so that $L'(n/s)n/s\sim L'(n)n$ as $n\ra\infty$, for every $s>0$. 
To justify the last step we can use
Potter's theorem (\cite{Binghametal}, Theorem 1.5.6) as before, 
to infer for every $\d>0$ the existence of a constant  $M>1$ such  that $L'(n/s)n/s/(L'(n)n)\lesssim s^\d\vee s^{-\d}$, for $s<n/M$ and $n\ge M$, so
that on this interval the integrand is dominated by a multiple of $(s^\d\vee s^{-\d})\bigl[ s^{-1}(1-e^{-s})s^{-\g}+\sum_{m=1}s^{m-\g}e^{-s}/(m+1)!\bigr]\lesssim s^{-\g-\d}\wedge s^{-1-\g+\d}$,
which is integrable for sufficiently small $\d>0$. Furthermore, for $s>n/M$, the function $|L'(n/s)n/s|$ is uniformly  bounded, whence the integral
over the interval $[n/M,\infty)$ is bounded above by a multiple of $\int_{n/M}^\infty s^{-1-\g}\,ds\lesssim n^{-\g}\ll L'(n)n$.

By the
identities obtained in the beginning of the proof, the integral in the right of the preceding display is identical to
$\Gamma(1-\g)/\g^2+\Gamma(1-\g)/\g=\Gamma(1-\g)(1+\g)/\g^2$.

Combining the preceding, we find that $E_n(\g)=n^\g\bigl(\bar H_n(\g)+O(n^{-\g})\bigr)+O(n^\b)
=n^\g L'(n)n+O(n^\b)$ and hence $E_n(\g)/\a(n)= O(n^\g L'(n)n/\a(n))+ O(n^\b/\a(n))$.
\end{proof}

\begin{lemma}
\label{LemmaInvariancenkn}
Suppose that $\a(u):=\#\{ j: 1/p_j\le u\}$ is regularly varying at $\infty$ of order $\g\in (0,1)$. Then
for any $\s_n\ra\s\in (0,1)$, and independent  $M_{n,j}\sim\text{Binomial}(n,p_j)$, 
and $k_n\ge n$ with $k_n-n=O(\sqrt n)$,
\begin{align*}
\sum_{j=1}^\infty\E (1_{M_{k_n,j}\ge 1}-1_{M_{n,j}\ge 1})&=o\bigl(\a(n)^{1/2}\bigr),
\end{align*}
Furthermore, if there exists a continuously differentiable function $L: [1,\infty)\to \RR$ such that
$|\a(u)-u^\g L(u)|\le C u^\b$, for every $u>1$, and some $C>0$ and $\b<\g$,
and $|L'(u)|\le C_\d u^{-1+\d}$, for every $u>1$ and $\d>0$ and some $C_\d>0$, then 
\begin{align*}
\sum_{j=1}^\infty\E\bigl (g_{\s_n}(M_{k_n,j})-g_{\s_n}(M_{n,j})\bigr)&=o\bigl(\a(n)^{1/2}\bigr).
\end{align*}
\end{lemma}

\begin{proof}
Because $\Pr(M_{n,j}=0)=(1-p_j)^n$, the left side of the first assertion is equal to 
$$\sum_{j=1}^\infty \bigl((1-p_j)^n-(1-p_j)^{k_n}\bigr)=\int_1^\infty\Bigl(1-\frac 1 u\Bigr)^n\Bigl( 1-\Bigl(1-\frac 1u\Bigr)^{k_n-n}\Bigr)\,d\a(u).$$
By the inequalities $1-x\le e^{-x}$, for $x\in\RR$, and $1-(1-x)^r\le r x$, for $x\in [0,1]$ and $r\in\NN$, this is bounded above by
$$\int_1^\infty e^{-n/u}(k_n-n)\frac 1u\,d\a(u)=\frac{k_n-n}n \int_0^n\a\Bigl(\frac ns\Bigr)e^{-s}(1-s)\,ds,$$
by Fubini's theorem, since $e^{-n/u}(n/u)=\int_0^{n/u} e^{-s}(1-s)\,ds$.
As in the proof of Lemma~\ref{LemConvergenceExpectations}, the integral is $\a(n)\bigl(\Gamma(1-\g)-\Gamma(2-\g)\bigr)(1+o(1))$.
Therefore, the preceding display divided by $\a(n)^{1/2}$ is of the order $\a(n)^{1/2}(k_n-n)/n\sim \a(n)^{1/2}n^{-1/2}$.
This tends to zero, as for every $\d>0$ we have that $\a(n)\le n^{\g+\d}$ eventually, by Potter's theorem, where $\g<1$ by assumption.

To prove the second assertion we first write 
\begin{align*}
\sum_{j=1}^\infty\E g_{\s_n}(M_{n,j})&=\sum_{j=1}^\infty \sum_{m=2}^ng_\s(m)\binom n m p_j^m(1-p_j)^{n-m}\\
&=\sum_{m=2}^n g_\s(m)\binom n m\int_1^\infty \Bigl(\frac 1u\Bigr)^{m}\Bigl(1-\frac 1u\Bigr)^{n-m}\,d\a(u).
\end{align*}
Writing $(1/u)^m(1-1/u)^{n-m}=\int_0^{1/u}s^{m-1}(1-s)^{n-m-1}(m-ns)\,ds$ (for $m>0$) and 
applying Fubini's theorem, we can rewrite this as 
\begin{align*}
&\sum_{m=2}^n g_\s(m)\binom n m\int_0^1\a\Bigl(\frac 1s\Bigr) s^{m-1}(1- s )^{n-m-1}(m-ns)\,ds\\
&\qquad=\int_0^1\sum_{l=1}^{n-1}\frac  1{l-\s}\sum_{m=l+1}^n\binom n m s^{m-1}(1- s)^{n-m-1}(m-ns)\,\a\Bigl(\frac 1s\Bigr)\,ds\\
&\qquad=\int_0^1\sum_{l=1}^{n-1}\frac {n-l}{l-\s} \binom n l s^{l}(1-s)^{n-l-1}\,\a\Bigl(\frac 1s\Bigr)\,ds,
\end{align*}
by Lemma~\ref{LemmaBinomialIdentity}. Thus the left side of the second assertion can be written in the form, with $k=k_n$,
\begin{align}
&\int_0^1\sum_{l=1}^{n-1}\frac {s^l(1-s)^{n-l-1}}{l-\s} \biggl[(k-l)\binom k l(1-s )^{k-n}-(n-l)\binom n l\biggr]\,\a\Bigl(\frac 1s\Bigr)\,ds\nonumber\\
&\qquad\qquad +\int_0^1\sum_{l=n+1}^k\frac {k-l}{l-\s} \binom k l s^{l}(1-s)^{k-l-1}\,\a\Bigl(\frac 1s\Bigr)\,ds.\label{EqDifferencekl}
\end{align}
Because $\a(1/s)\le 1/s$, the second term is bounded above by $\sum_{l>n}(k-l)/(l-\s)\binom k l B(l,k-l)$, for
$B$ the beta function. This is further bounded by $k\sum_{l>n} 1/((l-\s)l)\lesssim k/n\lesssim 1$.

By the assumption that $|\a(u)-u^\g L(u)|\le C u^\b$, if in the first term, we replace $\a(1/s)$ by $s^{-\g}L(1/s)$, the error
is bounded above by a multiple of
$$\int_0^1\sum_{l=1}^{n-1}\frac {s^l(1-s)^{n-l-1}}{l-\s} \Bigl|(k-l)\binom k l(1-s )^{k-n}-(n-l)\binom n l\Bigr|\,s^{-\b}\,ds.$$
The sum of the terms with $l>\sqrt n$ is bounded above by $a_{k,n}+a_{n,n}$, for
$$a_{k,n}=\sum_{l>\sqrt n}\frac{k-l}{l-\s}\binom kl B(l-\b+1,k-l)
\lesssim \sum_{l>\sqrt n}\frac{\Gamma(l-\b+1)}{(l-\s)l!}\frac{k!}{\Gamma(k-\b+1)}.$$
In view of Lemma~\ref{LemmaStirling}, $a_{k,n}$ 
is bounded above by a multiple of $\sqrt n^{-\b} k^\b=O(n^{\b/2})=o\bigl(\a(n)^{1/2}\bigr)$.
In the sum of the terms with $l\le \sqrt n$, we decompose $k-l=(k-n)+ (n-l)$ and
$\binom kl=\sum_i\binom n{l-i}\binom{k-n}i$, and bound
\begin{align*}
&\Bigl|(k-l)\binom k l(1-s )^{k-n}-(n-l)\binom n l\Bigr|
\le (k-n)\binom k l (1-s)^{k-n}\\
&\qquad\qquad\qquad+(n-l)\Bigl[\binom nl \bigl(1-(1-s)^{k-n}\bigr)+\sum_{i\ge 1}\binom n{l-i}\binom{k-n}i\Bigr].
\end{align*}
The middle term in the right is bounded above by $(n-l)\binom nl (k-n)s$. Thus the sum of the terms
with $l\le \sqrt n$ is bounded above by the sum of the three integrals
\begin{align*}
&\int_0^1 \sum_{l \le \sqrt n} \frac{s^l (1 - s) ^{n - l - 1}}{l - \sigma} (k - n) \binom k l (1 - s)^{k-n} s^{-\beta}  d s\\
&=\sum_{l\le \sqrt n} \frac{B(l+1-\b,k-l)}{l-\s}(k-n)\binom k l
=\sum_{l\le \sqrt n} \frac{\Gamma(l+1-\b)}{(l-\s)l!} \frac{\Gamma(k-l)}{(k-l)!} \frac{(k-n)k!}{\Gamma(k+1-\b)}\\
&\qquad\lesssim \sum_{l\le \sqrt n}\frac1{l^{1+\b}}\frac{k-n}{k-l}k^\b\lesssim n^{\b-1/2}\le n^{\b/2},\\
&\int_0^s \sum_{l \le \sqrt n} \frac{s^l (1 - s)^{n - l - 1}}{l - \sigma} (n - l) \binom n l (k -n) s s^{-\beta} ds\\
&=\sum_{l\le \sqrt n} \frac{B(l+2-\b,n-l)}{l-\s}(k-n)(n-l)\binom n l
=\sum_{l\le \sqrt n} \frac{\Gamma(l+2-\b)}{(l-\s)l!} \frac{(k-n)n!}{\Gamma(n+2-\b)}\\
&\qquad\lesssim \sqrt n^{1-\b}(k-n)n^{\b-1}\lesssim n^{\b/2},\\
&\int_0^1 \sum_{l \le \sqrt n} \frac{ s^l (1 - s)^{n - l - 1}}{l - \sigma} (n - l) \sum_i \binom{n}{l - i}\binom{k - n}{i} s^{-\beta} ds\\
&=\sum_{l\le \sqrt n}\sum_{i\ge 1} \frac{B(l+1-\b,n-l)}{l-\s}(n-l)\binom n{l-i}\binom{k-n} i\\
&\qquad=\sum_{l\le \sqrt n}\sum_{i\ge 1} \frac{\Gamma(l+1-\b)}{(l-\s)(l-i)!} \frac{(n-l)!}{(n-l+i)!}\frac{n!}{\Gamma(n+1-\b)}\binom{k-n}i\\
&\qquad \lesssim \sum_{l\le \sqrt n}\sum_{i\ge 1} l^{i-1-\b}\frac{1}{(n/2)^i}n^\b\binom{k-n}i
\le\sum_{i\ge 1}\frac{\sqrt n^{i-\b}}{(n/2)^i}n^\b\binom{k-n}i\\
&\qquad\le \Bigl(1+\frac 2{\sqrt n}\Bigr)^{k-n}n^{\b/2}
\lesssim n^{\b/2}.
\end{align*}
We conclude that replacing $\a(1/s)$ by $s^{-\g}L(1/s)$ in the first part of \eqref{EqDifferencekl} makes a difference of at most
of the order $n^{\b/2}=o\bigl(\a(n)^{1/2}\bigr)$. Finally, we consider the expression
\begin{align*}
&\int_0^1\sum_{l=1}^{n-1}\frac {s^l(1-s)^{n-l-1}}{l-\s} \biggl[(k-l)\binom k l(1-s )^{k-n}-(n-l)\binom n l\biggr]\,s^{-\g}L\Bigl(\frac 1s\Bigr)\,ds\\
&\qquad=\sum_{l=1}^{n-1}\frac{\Gamma(l+1-\g)}{(l-\s)l!}\Big[\frac{k!}{\Gamma(k+1-\g)}\E L(1/S_{l,k})-
\frac{n!}{\Gamma(n+1-\g)}\E L(1/S_{l,n})\Bigr],
\end{align*}
where $S_{l,k}$ is a random variable with the beta distribution with parameters $l+1-\g$ and $k-l$. 
The bound $|L'(u)|\le C_\d u^{-1+\d}$ gives that $L(u)\lesssim u^\d$, and hence
$|\E L(1/S_{l,k})|\lesssim \E S_{l,k}^{-\d}=B(l+1-\g-\d,k-l)/B(l+1-\g,k-l)\lesssim k^\d$. Therefore, after bounding the
difference with the help of the triangle inequality, 
the sum of the terms with $l> m$, can be bounded by $b_{k,m}+b_{n,m}$, for 
$$b_{k,m}=\sum_{l>m}\frac{\Gamma(l+1-\g)}{(l-\s)l!}\frac{k!}{\Gamma(k+1-\g)} k^\d
\lesssim \Bigl(\frac k m\Bigr)^\g k^\d.$$
For $m=n^{1/2+\d/\g+\eta}$, the right side is of the order $n^{\g/2-\eta \g}=o\bigl(\a(n)^{1/2}\bigr)$, for any $\eta>0$.
The sum of the terms with $l\le m$ is bounded above by
\begin{align*}
&\sum_{l=1}^m\frac{\Gamma(l+1-\g)}{(l-\s)l!}\Big[\frac{k!}{\Gamma(k+1-\g)}-\frac{n!}{\Gamma(n+1-\g)}\Bigr]\E L(1/S_{l,k})\\
&\qquad\qquad\qquad\qquad\qquad\qquad\qquad+ \frac{n!}{\Gamma(n+1-\g)}\Bigl[\E L(1/S_{l,k})-\E L(1/S_{l,n})\Bigr]\\
&\qquad\lesssim
\sum_{l=1}^m \frac1{l^{1+\g}}\Bigl[k^\g-n^\g+O\Bigl(\frac1n\Bigr)\Bigr]k^\d
+\sum_{l=1}^m\frac{n^\g}{l^{1+\g}}\Bigl|\E L(1/S_{l,k})-\E L(1/S_{l,n})\Bigr|.
\end{align*}
Here $k^\g-n^\g=n^\g\bigl((1+(k-n)/n)^\g-1\bigr)\lesssim n^{\g-1/2}$, so that the
first term is of the order $n^{\g-1/2}k^\d=o\bigl(\a(n)^{1/2}\bigr)$, for sufficiently small $\d>0$,
and hence is asymptotically negligible. For the second term we represent $S_{l,k}$ and $S_{l,n}$ using independent, gamma variables
$\bar\Gamma_l$, $\Gamma_{n-l}$ and $\Gamma_{k-n}$ with shape parameters $l+1-\g$, $n-l$ and $k-n$, and write
$|\E L(1/S_{l,k})-\E L(1/S_{l,n})|$ as
\begin{align*}
&\Bigl|\E L\Bigl(\frac{\bar\Gamma_l+\Gamma_{n-l}+\Gamma_{k-n}}{\bar\Gamma_l}\Bigr)
-\E L\Bigl(\frac{\bar\Gamma_l+\Gamma_{n-l}}{\bar\Gamma_l}\Bigr)\Bigr|\\
&\quad=\Bigl|\E \int _0^{\Gamma_{k-n}} L'\Bigl(\frac{\bar\Gamma_l+\Gamma_{n-l}+u}{\bar\Gamma_l}\Bigr)\frac1{\bar\Gamma_l}\,du\Bigr|\\
&\quad\lesssim \int_0^\infty\Pr(\Gamma_{k-n}>u)\E \frac{1}{(\bar\Gamma_l+\Gamma_{n-l}+u)^{1-\d}}\frac1{\bar\Gamma_l^\d}\,du
\le \E \Gamma_{k-n} \E  \frac{1}{\Gamma_{n-l}^{1-\d}}\E \frac1{\bar\Gamma_l^\d}.
\end{align*}
The three expecations can be computed explicitly in terms of the gamma function. Substituting the resulting
expressions in the second sum of
 the second last display, we see that this is bounded above by 
$$\sum_{l=1}^m\frac{n^\g(k-n)}{l^{1+\g}} \frac{\Gamma(n-l-1+\d)}{\Gamma(n-l)}\frac{\Gamma(l+1-\g-\d)}{\Gamma(l+1-\g)}
\lesssim  n^\g(k-n)\sum_{l=1}^m  \frac{1}{l^{1+\g+\d}(n-l)^{1-\d}}.$$
Since the summation indices satisfy $l\le m\ll n$, so that $n-l\gtrsim n/2$,
this is of the order $n^{-1/2+\g+\d}= o\bigl(\a(n)^{1/2}\bigr)$, for sufficiently small $\d$.
\end{proof}

\subsection{Supporting lemmas}

\begin{lemma}
\label{LemmaPoissonization}
Suppose that $V_{k,n}$, for $k,n\in\NN$, 
 are random variables independently of random variables $N_n\sim \text{Poisson}(n)$ so that
$a_n(V_{N_n,n}-E_n)\weak N(0,\tau^2)$ and $a_n(V_{k_n,n}-V_{n,n})\ra0$ in probability for every $k_n$ with $|k_n-n|=O(\sqrt n)$,
for $n\ra\infty$ and given numbers $a_n$ and $E_n$. Then $a_n(V_{n,n}-E_n)\weak N(0,\tau^2)$.
\end{lemma}

\begin{proof}
For any Lipschitz function $h: \RR\to[0,1]$ and $k_n$ as given, as $n\ra\infty$,
$$\Bigl|\E h\bigl(a_n(V_{k_n,n}-E_n)\bigr)-\E h\bigl(a_n(V_{n,n}-E_n)\bigr)\Bigr|
\le \E a_n|V_{k_n,n}-V_{n,n}|\wedge 1\ra 0.$$
By the central limit theorem the probability $\Pr\bigl(|N_n-n|>\sqrt n M\bigr)$ can be made arbitrarily small
uniformly in $n$ by choosing sufficiently large $M$. Then $\E h\bigl(a_n(V_{n,n}-E_n)\bigr)$ is arbitrarily close
to 
\begin{align*}
&\E h\bigl(a_n(V_{n,n}-E_n)\bigr)\sum_{k: |k-n|\le \sqrt n M}\Pr(N_n=k)\\
&\qquad\qquad=\sum_{k: |k-n|\le \sqrt n M}\E h\bigl(a_n(V_{k,n}-E_n)\bigr)\Pr(N_n=k)+o(1),\end{align*}
by the preceding display, as $n\ra\infty$, for every fixed $M$. The sum in the right side is arbitrarily close to 
$\E h\bigl(a_n(V_{N_n,n}-E_n)\bigr)$ uniformly in $n$, if $M$ is chosen sufficiently large, which tends to $\E h(\tau Z)$, for $Z\sim N(0,1)$
as $n\ra\infty$, by assumption. We conclude that $\E h\bigl(a_n(V_{n,n}-E_n)\bigr)$ is arbitrarily close to $\E h(\tau Z)$, as $n\ra\infty$.
\end{proof}

\begin{lemma}
\label{LemmaLindeberg}
If $X_{n,1},X_{n,2},\ldots$ are independent random variables with $s_n^2:=\sum_{j=1}^\infty\var X_{n,j}<\infty$ and
$s_n^{-2}\sum_{j=1}^\infty\E X_{n,j}^21_{|X_{n,j}|>\e s_n}\ra 0$, for every $\e>0$, then
$\sum_{j=1}^\infty (X_{n,j}-\E X_{n,j})/s_n\weak N(0,1)$.
\end{lemma}

\begin{proof}
The variables $Y_{n,j}=(X_{n,j}-\E X_{n,j})/s_n$ have mean zero and $\sum_{j}\E Y_{n,j}^2=1$. Choose 
integers $k_n\uparrow\infty$ such that  $\sum_{j\le k_n}\E Y_{n,j}^2\uparrow 1$. 
Then the sequence $\sum_{j>k_n}Y_{n,j}$ tends to zero in
second mean and hence it suffices to show that $\sum_{j\le k_n}Y_{n,j}\weak N(0,1)$. The latter follows from the
Lindeberg central limit theorem provided that 
$\sum_{j\le k_n}\E Y_{n,j}^21_{|Y_{n,j}|>\e}\ra 0$, for every $\e>0$.
To see that this is satisfied, first note that the Lindeberg condition implies that $\max_{j}\E X_{n,j}^2/s_n^2\ra0$ and hence both 
$s_n^{-2}\E \sum_j |\E X_{n,j}|^21_{|X_{n,j}|>\e s_n}\le o(1)\sum_j \Pr(|X_{n,j}|>\e s_n)\ra 0$ and
$|\E X_{n,j}|\le \e s_n$, for every $j$ eventually, for every fixed $\e>0$. This shows that
the variables $X_{n,j}$ also satisfy the centered ``infinite Lindeberg
condition'' $\sum_j\E Y_{n,j}^21_{|Y_{n,j}|>\e}\ra 0$, for every $\e>0$, which implies
the Lindeberg condition for the finite array $Y_{n,1},\ldots, Y_{n,k_n}$.
\end{proof}

\begin{lemma}
\label{LemmaBinomialIdentity}
For every $p\in [0,1]$ and $l\in\NN\cup\{0\}$ and $n\in \NN$,
$$\sum_{m=l+1}^n\binom n m p^{m-1}(1-p)^{n-m-1}(m-np)=(n-l)\binom nlp^l(1-p)^{n-l-1}.$$
\end{lemma}

\begin{proof}
For $X_{n-1}$ and $X_n$ the numbers of successes in the first $n-1$ and $n$ independent Bernoulli
trials with success probability $p$,  we have $\{X_n\ge l+1\}\subset\{X_{n-1}\ge l\}$ and
$\{X_{n-1}\ge l\}-\{X_n\ge l+1\}=\{X_{n-1}=l, B_n=0\}$, for $B_n$ the outcome of the $n$th trial.
This gives the identity $\Pr(X_{n-1}\ge l)-\Pr(X_n\ge l+1)= \Pr(X_{n-1}=l)(1-p)$.
We multiply this by $n/(1-p)$ to obtain the identity given by the lemma, which we first rewrite using
that $m\binom n m=n\binom {n-1}{m-1}$ and $(n-l)\binom n l= n\binom {n-1} l$.
\end{proof}

\begin{lemma}
\label{LemmaStirling}
For every $\g\in(0,1)$, as $n\ra\infty$,
$$\frac{\Gamma(n-\g)n^\g}{\Gamma(n)}=1+O\Bigl(\frac1 n\Bigr).$$
\end{lemma}

\begin{proof}
By Stirling's approximation, the quotient is
\begin{align*}\frac{\sqrt{{2\pi}/{(n-\g)}} \left(\frac{n-\g}e\right)^{n-\g}\bigl(1+O(1/n)\bigr)n^\g}
{\sqrt{{2\pi}/{n}} \left(\frac{n}e\right)^{n}\bigl(1+O(1/n)\bigr)}
&=\Bigl(\frac{n-\g}n\Bigr)^ne^\g\Bigl(1+O\Bigl(\frac1n\Bigr)\Bigr)\\
&=\Bigl(e^{-\g}+O\Bigl(\frac1n\Bigr)\Bigr)e^\g\Bigl(1+O\Bigl(\frac1n\Bigr)\Bigr).
\end{align*}
\end{proof}

\begin{lemma}
\label{LemmaRV}
If  $\a: (1,\infty)\to \RR$ satisfies $|\a(u)-u^\g L(u)|\le C u^\b$, for $u>1$ and $\b<\g$ and a slowly varying function $L$, 
then $\a$ is regularly varying of order $\g$.
\end{lemma} 

\begin{proof}
Since $L$ is slowly varying, we have $L(n)\gtrsim n^{-\d}$, for every $\d>0$, and hence
$n^\g L(n)\gtrsim n^{\g-\d}\gg n^\b$, for sufficiently small $\d>0$. Since 
$|\a(n)-n^\g L(n)|\lesssim n^\b$ by assumption, it follows that $\a(n)\gg n^\b$ and $n^\g L(n)/\a(n)\ra 1$,
as $n\ra\infty$. By the assumption we have $\a(nu)=(nu)^\g L(nu)+O(n^\b)
=(nu)^\g L(n)(1+o(1))+O(n^\b)$, for every $u$ as $n\ra\infty$, since $L$ is slowly varying,
and hence $\a(nu)/\a(n)=u^\g \bigl(n^\g L(n)/\a(n)\bigr)(1+o(1))+O(n^\b/\a(n))\ra u^\g$.
\end{proof}

\begin{lemma}
\label{LemmaMisc}
For every $s\ge 0$ and $\d\in (0,1]$:
\begin{itemize}
\item[(i)] $\sum_{m=1}^\infty s^mm^\d/m!\le s^\d e^s$, 
\item[(ii)] $\sum_{m=2}^\infty s^mm^\d/m!\le s^\d (e^s-1)$.
\end{itemize}
\end{lemma}

\begin{proof}
Since $s^mm^\d/m!=  (s^mm/m!)^\d(s^m/m!)^{1-\d}$, H\"older's inequality with $p=1/\d$ and $q=1/(1-\d)$ gives that
the sums on the left are bounded above by $\bigl(\sum_m (s^mm/m!)\bigr)^\d\bigl(\sum_m (s^m/m!)\bigr)^{1-\d}$,
where the summation starts at $m=1$ for (i) and at $m=2$ for (ii).
In the case of (i) the first series is bounded above by $s e^s$ and the second by $e^s-1$,
while in the case of (ii) the bounds $s (e^s-1)$ and $e^s-1-s$ pertain.
\end{proof}

\begin{lemma}
\label{LemmaSum}
For $K\ra\infty$, we have $\sum_{i=1}^{K-1}\log (M+i\s)=K\log K+K\log (\s/e)+(M/\s-1/2)\log K+\log(\sqrt{2\pi}/\s)-\log\Gamma(1+M/\s)+O(1/K)$, 
where the remainder is bounded above by a universal multiple of $(M/\s+1)^2/K$, for all $M\ge 0,\s>0$.
\end{lemma}

\begin{proof}
The sum is equal to $(K-1)\log \s+\log \Gamma(K+M/\s)-\log \Gamma(1+M/\s)$. By the expansion for the log Gamma function,
the middle term can be expanded as
$$\log \Gamma\Bigl(K+\frac M\s\Bigr)=\Bigl(K+\frac M\s-\frac12\Bigr)\log \Bigl(K+\frac M\s\Bigr)-\Bigl(K+\frac M\s\Bigr)+\log\sqrt{2\pi}+O\Bigl(\frac1{K}\Bigr),$$
where the remainder term is uniform in $M$ and $\s$.
Next expand $\log (K+M/\s)$ as $\log K+M/(\s K)+O((M/\s)^2/ K^2)$. Finally we collect terms.
\end{proof}

\bibliographystyle{acm}
\bibliography{BvMPYbib}

\end{document}